\documentclass[11pt,a4paper]{article}
\usepackage[utf8]{inputenc} 
\usepackage[T1]{fontenc} 
\usepackage[english]{babel}
\usepackage[style=alphabetic,maxnames=99,backend=biber]{biblatex}
\bibliography{dh}
\usepackage{mathtools, amssymb, amsthm} 
\DeclareMathOperator*{\diam}{diam}
\DeclareMathOperator*{\identity}{Id}

\DeclareMathOperator*{\meas}{meas}
\DeclareMathOperator*{\range}{range}
\usepackage{enumerate} 
\usepackage{graphicx}
\usepackage{subfig}
\usepackage[colorlinks=false]{hyperref}

\title{ Variational discretization of a
        control-constrained parabolic bang-bang
           optimal control problem
}
\author{
Nikolaus von Daniels%
    \footnote{Schwerpunkt Optimierung und Approximation,
    Universität Hamburg, Bundesstraße~55, 20146~Hamburg, Germany,
    \texttt{nvdmath@gmx.net}, \texttt{michael.hinze@uni-hamburg.de}
    }
\and Michael Hinze\footnotemark[1]
}

\date{\today}

\begin{document}

\maketitle

\newtheorem{theo}{Theorem}
\newtheorem{algo}[theo]{Algorithm}
\newtheorem{defi}[theo]{Definition}
\newtheorem{lemm}[theo]{Lemma}
\newtheorem{coro}[theo]{Corollary}
\newtheorem{assu}[theo]{Assumption}
\newtheorem{rema}[theo]{Remark}
\newtheorem{exam}[theo]{Example}

\DeclarePairedDelimiter\abs{\lvert}{\rvert}
\DeclarePairedDelimiter\norm{\lVert}{\rVert}
\newcommand{\twoset}[2]{\ensuremath{\left\{#1\,\left|\;#2\right.\right\}}}
\newcommand{\dual}[1]{\ensuremath{{#1}^*}}
\newcommand{\Uad}{U_\textup{ad}}
\newcommand{\uopt}{\ensuremath{\bar u}}
\newcommand{\yopt}{\ensuremath{\bar y}}
\newcommand{\popt}{\ensuremath{\bar p}}
\newcommand{\uoptd}{\ensuremath{\bar u_{kh}}}
\newcommand{\yoptd}{\ensuremath{\bar y_{kh}}}
\newcommand{\poptd}{\ensuremath{\bar p_{kh}}}
\newcommand{\V}{\ensuremath{H^1_0(\Omega)}}
\newcommand{\Vd}{\ensuremath{H^{-1}(\Omega)}}
\renewcommand{\H}{\ensuremath{L^2(\Omega)}}  
\newcommand{\Loo}{\ensuremath{L^\infty(\Omega)}} 
\newcommand{\HI}{\ensuremath{H^1(\Omega)}} 
\newcommand{\LIIH}{\ensuremath{L^2(I,\H)}}
\newcommand{\LIIV}{\ensuremath{L^2(I,\V)}}
\newcommand{\LIIVd}{\ensuremath{L^2(I,\Vd)}}
\newcommand{\LooLoo}{\ensuremath{\infty}}
\newcommand{\person}[1]{\textsc{#1}}
\newcommand{\restr}[2]{{
  \left.\kern-\nulldelimiterspace 
  #1 
  \vphantom{\big|} 
  \right|_{#2} 
}}

\noindent {\small {\bf Abstract:}
We consider a control-constrained parabolic optimal control problem 
without Tikhonov term in the tracking functional.
For the numerical treatment, we use
variational discretization of its Tikhonov regularization:
For the state and the adjoint equation, we apply
Petrov-Galerkin schemes from \cite{DanielsHinzeVierling}  in time 
and usual conforming finite elements in space.
We prove a-priori estimates for the error between the
discretized regularized problem and the limit problem. 
Since these estimates are not robust
if the regularization parameter tends to zero, 
we establish robust estimates, which 
--- depending on the problem's regularity ---
enhance the previous ones. 
In the special case of bang-bang solutions, these estimates are
further improved.
A numerical example confirms our analytical findings.
}\\[2mm]


\section{Introduction}
In this article we are interested in the numerical solution
of the optimal control problem
\begin{equation}\tag{$\mathbb P_0$}\label{OCPl}
   \min_{u\in\Uad} J_0(u)\quad\text{with}\quad
          J_0(u):= \frac{1}{2} \norm{Tu-z}^2_{H}.
\end{equation}
Here, $T$ is basically the
(weak) solution operator of the heat equation,
the set of admissible controls $\Uad$ is given by box constraints,
and $z\in H$ is a given function to be tracked. 

Often, the solutions of \eqref{OCPl} possess a special structure:
They take values only on the bounds of the admissible set $\Uad$ 
and are therefore called \emph{bang-bang solutions}.

Theoretical and numerical questions related to this control problem
attracted much interest in recent years, see, e.g., 
\cite{deckelnick-hinze},
\cite{wachsmuth1},
\cite{wachsmuth2},
\cite{wachsmuth3},
\cite{wachsmuth4},
\cite{wachsmuth5},
\cite{gong-yan},
\cite{felgenhauer2003},
\cite{alt-bayer-etal2},
\cite{alt-seydenschwanz-reg1},
and
\cite{seydenschwanz-regkappa}. 
The last four papers are concerned with $T$ being the solution operator
of an \emph{ordinary} differential equation, the former papers with $T$ 
being a solution operator of an \emph{elliptic} PDE 
or $T$ being a continuous linear operator.
In \cite{dissnvd}, a brief survey of the content of these and some
other related papers is given at the end of the bibliography.

Problem~\eqref{OCPl} is in general ill-posed, meaning that 
a solution does not depend continuously on the datum $z$,
see \cite[p. 1130]{wachsmuth2}.
The numerical treatment of a discretized version of \eqref{OCPl} is also
challenging, e.g., due to the absense of formula \eqref{FONC} in the case
$\alpha=0$, which corresponds to problem \eqref{OCPl}.

Therefore we use Tikhonov regularization to overcome
these difficulties. The \emph{regularized problem} is given by
\begin{equation}\tag{$\mathbb P_\alpha$}\label{OCPr}
   \min_{u\in\Uad} J_\alpha(u)\quad\text{with}\quad
          J_\alpha(u):= \frac{1}{2} \norm{Tu-z}^2_{H}
                              + \frac{\alpha}{2} \norm{u}^2_{U}
\end{equation}
where $\alpha > 0$ denotes the regularization parameter.
Note that for $\alpha=0$, problem \eqref{OCPr} reduces to problem
\eqref{OCPl}.

For the numerical treatment of the regularized problem,
we then use variational discretization introduced 
by Hinze in \cite{Hinze2005}, see also \cite[Chapter 3.2.5]{hpuu}.
The state equation is treated with a Petrov-Galerkin
scheme in time using a piecewise constant Ansatz for the state 
and piecewise linear, continuous test functions.
This results in variants of the Crank-Nicolson scheme for 
the discretization of the state and the 
adjoint state, which were proposed recently in \cite{DanielsHinzeVierling}.
In space, usual conforming finite elements are taken.
See \cite{dissnvd} for the fully discrete case and
\cite{SpringerVexler2013}
for an alternative discontinuous Galerkin approach. 

The purpose of this paper is to prove a-priori bounds for 
the error between the discretized regularized problem and the
limit problem, i.e. the continuous unregularized problem.

We first derive error estimates between
the discretized regularized problem and its continuous counterpart.
Together with Tikhonov error estimates 
recently obtained in \cite{daniels}, see also \cite{dissnvd},
one can
establish estimates for the total error between the discretized regularized
solution and the solution of the continous limit problem, i.e. $\alpha =0$.
Here, second order convergence in space is not achievable and 
(without coupling) the estimates are not robust if $\alpha$ tends to zero.
Using refined arguments, we overcome both drawbacks.
In the special case of bang-bang controls, we further improve 
those estimates.

The obtained estimates suggest a coupling rule for the parameters 
$\alpha$ (regularization parameter),
$k$, and $h$ (time and space discretization parameters, respectively)
to obtain optimal convergence rates which we numerically observe.

The paper is organized as follows. 

In the next section, we introduce the
functional analytic description of the regularized problem. We recall
several of its properties, such as
existence of a unique solution for all $\alpha \ge 0$ (thus especially
in the limit case $\alpha=0$ we are interested in), an explicit 
characterization of the solution structure, and the function space
regularity of the solution.
We then introduce the Tikhonov regularization and recall
some error estimates under suitable assumptions. In the special case
of bang-bang controls, we recall a smoothness-decay lemma which later
helps to improve the error estimates for the discretized problem.

The third section is devoted to the discretization of
the optimal control problem.
At first, the discretization of the state and adjoint equation
is introduced and several error estimates needed in the later analysis 
are recalled.
Then, the analysis of variational discretization
of the optimal control problem is conducted.

The last section discusses a numerical example where we observe the
predicted orders of convergence.

\section{The continuous optimal control problem}

\subsection{Problem setting and basic properties}
Let $\Omega\subset\mathbb{R}^d$, $d\in\{2,3\}$, be a spatial domain
which is assumed to be bounded and convex with a polygonal boundary 
$\partial\Omega$. Furthermore, a fixed time interval 
$I:=(0,T)\subset\mathbb R$, $0<T<\infty$, a desired state 
$y_d\in \LIIH$, a non-negative real constant $0 \le \alpha \in \mathbb R$,
and an initial value $y_0\in\H$ are prescribed.
With the Gelfand triple
$\V \hookrightarrow \H \hookrightarrow \Vd$
we consider the following optimal control problem
\begin{equation}\label{OCP}\tag{$\mathbb P$}
\begin{aligned}
&\min_{ y\in Y,u\in\Uad} J(y,u)\quad\text{with}\quad J(y,u):= 
       \frac{1}{2}\norm{y-y_d}^2_{\LIIH}+
            \frac{\alpha}{2}\norm{u}^2_U,\\
&\text{s.t. } y=S(Bu,y_0)
\end{aligned}
\end{equation}
where $U:=L^2(\Omega_U)$ is the control space,
the (closed and convex) set of admissible controls is defined by
\begin{equation} \label{E:Uad}
   \Uad:=\twoset{u\in U}{ a(x)\le u(x)\le b(x) 
                 \quad\forall' x\in\Omega_U}
\end{equation} 
with fixed control bounds $a$, $b$ $\in$ $L^\infty(\Omega_U)$
fulfilling $a\le b$ almost everywhere in $\Omega_U$,
\[ 
    Y:=W(I):=\twoset{v\in \LIIV}{v_t \in \LIIVd} 
\]
is the state space, and the control operator $B$ as well as the
control region $\Omega_U$ are defined below. 

Note that we use the notation $v_t$ and $\partial_t v$ 
for weak time derivatives and $\forall'$ for ``for almost all''.

The operator 
\begin{equation}\label{E:operatorS}
   S: \LIIVd\times \H \rightarrow W(I),\quad (f,g) \mapsto y:= S(f,g),
\end{equation}
denotes the weak solution operator associated with the heat equation,
i.e., the linear parabolic problem
\begin{equation*}
\begin{aligned} \partial_t y -\Delta y &= f &&\text{in }I\times\Omega\,,\\
y&=0&&\text{in } I\times\partial\Omega\,,\\
y(0)&=g&&\text{in } \Omega\,.
\end{aligned}
\end{equation*}
The weak solution is defined as follows. For $(f,g) \in \LIIVd\times \H$ 
the function $y\in W(I)$ with
$\langle \cdot,\cdot \rangle := \langle \cdot,\cdot \rangle_{\Vd\V}$
satisfies the two equations
\begin{subequations}\label{E:WF}
\begin{align}
 y(0)={}&g \\
\begin{split}
\int_0^T \bigg\langle \partial_t y(t),v(t)\bigg\rangle
   + a(y(t),v(t))\, dt
={}& \int_0^T \bigg\langle f(t),v(t)\bigg\rangle\, dt\\
    \phantom{=}{}&\quad\forall\, v\in \LIIV.
\end{split}
\end{align}
\end{subequations}
Note that by the embedding 
$W(I)\hookrightarrow C([0,T],\H)$, see, e.g., \cite[Theorem 5.9.3]{evans},
the first relation is meaningful.\\
In the preceding equation, the bilinear form 
$a:H^1(\Omega)\times H^1(\Omega)\to\mathbb R$ is given by
\[
a(f,g):= \int_\Omega \nabla f(x) \nabla g(x)\ dx.
\]
We show below that \eqref{E:WF} yields an operator $S$
in the sense of \eqref{E:operatorS}.

For the control region $\Omega_U$ and the
control operator $B$ we consider two situations.
\begin{enumerate}
\item (Distributed controls) 
We set $\Omega_U := I\times\Omega$, and define the
control operator $B:U\rightarrow \LIIVd$ by 
$B:=\identity$, i.e., the identity mapping induced by the
standard Sobolev embedding $\H\hookrightarrow\Vd$. 

\item (Located controls) 
We set the control region $\Omega_U := I$.
With a fixed functional $g_1\in\Vd$ the linear and continuous control
operator $B$ is given by
\begin{equation}\label{E:B}
B: U=L^2(I)\rightarrow \LIIVd\,,\quad u\mapsto 
      \left( t\mapsto u(t)g_1 \right).
\end{equation}
The case of $D$ fixed functionals $g_i$ with controls $u_i$ 
and a control operator
$B: L^2(I,\mathbb R^D)\rightarrow \LIIVd$,
$u\mapsto \left( t\mapsto \sum_{i=1}^D u_i(t)g_i \right)$
is a possible generalization.
To streamline the presentation we restrict ourselves to the case
$D=1$ here and refer to \cite{dissnvd} for the case $D>1$.

For later use we observe that the adjoint operator $\dual B$
is given by 
\begin{equation*}  
\dual B: \LIIV \to U=L^2(I), \quad
   (\dual Bq)(t) = \langle g_1,q(t)\rangle_{\Vd\V}.
\end{equation*}
If furthermore $g_1\in\H$ holds, we can consider
$B$ as an operator $B: L^2(I)\to \LIIH$ 
and get the adjoint operator 
\begin{equation*}   
\dual B:\LIIH\to U=L^2(I), \quad
   (\dual Bq)(t) = (g_1,q(t))_{\H}.
\end{equation*}
Note that the adjoint operator $\dual B$
(and also the operator itself) is preserving time regularity, i.e.,
$\dual B : H^k(I,X)\to H^k(I)$ for $k\ge 0$
where $X$ is a subspace of $\H$ depending on the regularity of the $g_1$
(as noticed just before), e.g., $X=\H$ or $X=\V$.
\end{enumerate}

\begin{lemm}[Properties of the solution operator $S$]
\mbox{}
\begin{enumerate}
\item For every 
$(f,g) \in \LIIVd\times \H$ a unique state $y \in W(I)$ 
satisfying \eqref{E:WF} exists. Thus the operator $S$ from
\eqref{E:operatorS} exists. Furthermore the state fulfills
\begin{equation}\label{E:stabS}
\norm{y}_{W(I)} \le C 
    \left(\norm{f}_{\LIIVd}+\norm{g}_{\H}\right).
\end{equation}
\item Consider the bilinear form $A:W(I)\times W(I)\to\mathbb R$
given by
\begin{equation}\label{bilinA}
A(y,v):= \int_0^T -\bigg\langle v_t,y\bigg\rangle + 
 a(y,v)\, dt + \bigg\langle y(T),v(T)\bigg\rangle
\end{equation}
with $\langle \cdot,\cdot \rangle := \langle \cdot,\cdot \rangle_{\Vd\V}$.
Then for $y\in W(I)$, equation \eqref{E:WF} is equivalent to
\begin{equation}\label{E:WFM}
A(y,v) = \int_0^T \bigg\langle f,v\bigg\rangle\, dt +
    (g,v(0))_{\H}\quad\forall\ v\in W(I).
\end{equation}
Furthermore, $y$ is the only function in $W(I)$ fulfilling
equation \eqref{E:WFM}.
\end{enumerate}
\end{lemm}
\begin{proof}
This can be derived using standard results, see \cite[Lemma~1]{dissnvd}.
\end{proof}
An advantage of the formulation \eqref{E:WFM} in comparison to 
\eqref{E:WF} is the fact that the weak time derivative $y_t$ of $y$ is 
\emph{not} part of the equation. Later in discretizations of this 
equation, it offers the possibility to consider states which do not
possess a weak time derivative.

We can now establish the existence of a solution to problem \eqref{OCP}.
\begin{lemm}[Unique solution of the o.c.p.]\label{L:OCPexistence}
\mbox{}\\
The optimal control problem \eqref{OCP} admits
for fixed $\alpha \ge 0$
a unique solution
$(\yopt_\alpha,\uopt_\alpha)\in Y\times U$, which can be characterized
by the first order necessary and sufficient optimality condition
\begin{equation}\label{VarIneq}
   \uopt_\alpha\in\Uad,\quad
   \left( \alpha\uopt_\alpha + \dual B\popt_\alpha,u-\uopt_\alpha\right)_U 
   \ge 0\quad \forall\ u\in\Uad,
\end{equation}
where $\dual B$ denotes the adjoint operator of $B$, 
and the so-called optimal adjoint state 
$\popt_\alpha \in W(I)$ is the 
unique weak solution 
defined and uniquely determined by the equation
\begin{equation}\label{E:AA}
A(v,\popt_\alpha)
=
\int_0^T \langle h,v\rangle_{\Vd\V}\, dt 
\quad \forall\ v\in W(I).
\end{equation}
\end{lemm}
\begin{proof}
This follows from standard results, see, e.g.,
\cite[Lemma~2]{dissnvd}.
\end{proof}
As a consequence of the fact that $\Uad$ is a closed and convex set in
a Hilbert space we have the following lemma.
\begin{lemm}\label{L:orthoproj}
In the case $\alpha > 0$ the variational inequality \eqref{VarIneq}
is equivalent to
\begin{equation}\label{FONC}
\uopt_\alpha = P_{\Uad}\left(-\frac{1}{\alpha}\dual B\popt_\alpha\right),
\end{equation}
where $P_{\Uad}: U \to \Uad$ is the orthogonal projection. 
\end{lemm}
\begin{proof}
See \cite[Corollary 1.2, p. 70]{hpuu} with $\gamma = \frac 1\alpha$.
\end{proof}
The orthogonal projection in \eqref{FONC} can be made explicit in our
setting.
\begin{lemm}\label{L:projexplicit}
Let us for $a,b\in\mathbb R$ with $a\le b$ consider the projection of a
real number $x\in\mathbb R$ into the interval $[a,b]$, i.e.,
$P_{[a,b]}(x):=\max\{a,\min\{x,b\}\}$.

There holds for $v\in U$
\[
P_{\Uad}(v)(x) = P_{[a(x),b(x)]}(v(x))\quad \forall' x\in\Omega_U.
\]
\end{lemm}
\begin{proof}
See \cite[Lemma~4]{dissnvd}
for a proof of this standard result in our setting.
\end{proof}
We now derive an explicit characterization of the optimal control.
\begin{lemm}\label{L:ucharact}
If $\alpha >0$, then for almost all $x\in \Omega_U$
there holds for the optimal control
\begin{equation}\label{E:uoptpwchar}
  \uopt_\alpha  (x)= \begin{cases}   
        a(x)&\text{if $\dual B\popt_\alpha (x)+\alpha a(x) >0$},\\
        -\alpha^{-1}\dual B\popt_\alpha (x)
           &\text{if $\dual B\popt_\alpha (x)
                     + \alpha \uopt_\alpha (x) = 0$},\\
        b(x)&\text{if $\dual B\popt_\alpha (x)+\alpha b(x) <0$}.
                      \end{cases}
\end{equation}

Suppose $\alpha = 0$ is given. Then the optimal control fulfills a.e.
\begin{equation}\label{E:bangbang}
  \uopt_0 (x)= 
     \begin{cases}   
        a(x)&\text{if $\dual B\popt_0 (x) >0$},\\
        b(x)&\text{if $\dual B\popt_0 (x) <0$}.
     \end{cases}
\end{equation}
\end{lemm}
\begin{proof}
We refer to \cite[Lemma~5]{dissnvd} for a proof of this standard
result in our setting.
\end{proof}

\begin{rema}\label{R:bangbang}
As a consequence of \eqref{E:bangbang} we have:
If $\dual B\popt_0$ vanishes only on a subset of
$\Omega_U$ with Lebesgue measure zero, the optimal control
$\uopt_0$ only takes values on the
bounds $a, b$ of the admissible set $\Uad$.
In this case $\uopt_0$ is called a \emph{bang-bang solution}.
\end{rema}

Assuming more regularity on the data than stated above,
we get regularity for the optimal state $\yopt_\alpha$
and the adjoint state $\popt_\alpha$ needed for the
convergence rates in the numerical realization of the problem.

We use here and in what follows the notation 
\[ 
   \norm{\cdot} := \norm{\cdot}_{\H}, 
   \norm{\cdot}_I := \norm{\cdot}_{\LIIH},
\]
\[   (\cdot,\cdot) := (\cdot,\cdot)_{\H}, \quad\text{and}\quad
   (\cdot,\cdot)_I := (\cdot,\cdot)_{\LIIH}. 
\]
 
\begin{assu}\label{A:Regularity}
Let $y_d\in H^2(I,\H)\bigcap H^1(I,H^2(\Omega)\cap\V)$ with
$\Delta y_d(T)$ $\in$ $\V$ and $y_0\in\V$.
Furthermore, we expect $\Delta y_0\in\V$.
In the case of distributed controls, we assume 
$a$, $b\in$ $H^1(I,\H)$ $\bigcap$ $C(\bar I,\V\cap C(\bar\Omega))$.
In the case of located controls, we assume $g_1\in\V$,
and $a,b\in W^{1,\infty}(I)$.
\end{assu}

\begin{lemm}[Regularity of problem \eqref{OCP}, $\alpha > 0$]
\label{L:regOCPposalpha}
Let Assumption~\ref{A:Regularity} hold and let $\alpha > 0$.
For the unique solution $(\yopt,\uopt)$ of \eqref{OCP} 
and the corresponding adjoint state $\popt$ there holds 
\begin{itemize}
\item
$ \popt\in H^3(I,\H) \bigcap H^2(I,H^2(\Omega)\cap\V) 
       \hookrightarrow C^2(\bar I,\V),
$
\item
$
\yopt\in H^2(I,\H) \bigcap H^1(I,H^2(\Omega)\cap\V)
       \hookrightarrow C^1(\bar I,\V),\quad\text{and}
$
\item
$
\uopt\in W^{1,\infty}(I)
$ 
 in the case of located controls or
\item
$
\uopt\in H^1(I,L^2(\Omega))\cap C(\bar I,\V)\cap 
          C(\bar I\times\bar\Omega)$
in the case of distributed controls.
\end{itemize}
With some constant $C>0$ independent of $\alpha$,
we have the a priori estimates
\begin{multline}
\label{E:uypapriori}
\norm{\partial^2_t \yopt}_I+\norm{\partial_t\Delta \yopt}_I
    + \max_{t\in [0,T]} \norm{\nabla \partial_t \yopt(t)} \\
\le d_1(\uopt) := C\left(\norm{B\uopt}_{H^1(I,\H)}
   +\norm{\nabla B\uopt(0)}
   +\norm{\nabla \Delta y_0}\right),\\
\shoveleft{
 \norm{\partial^2_t \popt}_I+\norm{\partial_t\Delta \popt}_I
    + \max_{t\in [0,T]} \norm{\nabla \partial_t \popt(t)}
} \\
  \le d_0(\uopt) := C \left(\norm{y_d}_{H^1(I,L^2(\Omega))}
          +\norm{\nabla y_d(T)}
          +\norm{B\uopt}_I+\norm{\nabla y_0}\right),\ \text{and}\\
\shoveleft{
\norm{\partial^3_t \popt}_I+\norm{\partial^2_t\Delta \popt}_I
       + \max_{t\in [0,T]} \norm{\nabla \partial^2_t \popt(t)}
}\\
\le d_1^+(\uopt) := d_1(\uopt) + \\
C \left(\norm{\partial^2_t y_d}_I
   +\norm{\nabla \partial_t y_d(T)} 
   + \norm{\nabla \Delta y_d(T)}  
   +\norm{\nabla B\uopt(T)} \right).
\end{multline}
\end{lemm}
\begin{proof}
See \cite[Lemma~12]{dissnvd}.
\end{proof}

\begin{rema}[Regularity in the case $\alpha=0$]
\label{R:regOCPalpha0}
In the case $\alpha = 0$, we have less regularity: 
\begin{itemize}
   \item $\popt \in H^1(I,H^2(\Omega)\cap\V)\bigcap H^2(I,\H)
                \hookrightarrow C^1(\bar I,\V)$, and
   \item $\yopt \in L^2(I,H^2(\Omega)\cap\V)\bigcap H^1(I,\H)
                \hookrightarrow C([0,T],\V).$
\end{itemize}
Since \eqref{FONC} does not hold, we can not
derive regularity for $\uopt$ from that of $\popt$ as above.
We only know from the definition of $\Uad$ that 
$\uopt \in L^\infty(\Omega_U)$, but might be 
discontinuous as we will see later. 
\end{rema}

\subsection{Tikhonov regularization}
For this subsection, it is useful to rewrite problem \eqref{OCP} 
in the reduced form \eqref{OCPr}
with $H:=\LIIH$, fixed data $z:=y_d-S(0,y_0)$ and the linear and
continuous control-to-state
operator $T : U \to H$, $Tu:=S(Bu,0)$.
%
%
From now onwards we assume 
\begin{equation}\label{E:boundssep}
    a \le 0 \le b
\end{equation}
in a pointwise almost everywhere sense 
where $a$ and $b$ are the bounds of the admissable set $\Uad$. 
For the limit problem \eqref{OCPl},
which we finally want to solve, this assumption
can always be met by a simple transformation of the variables.

To prove rates of convergence with respect to $\alpha$, we rely on
the following assumption.
\begin{assu}\label{A:sourcestruct}
    There exist a set $A\subset \Omega_U$, a function $w\in H$ with
    $\dual Tw\in L^\infty(\Omega_U)$, and 
    constants $\kappa > 0$ and $C\ge 0$, such that there holds 
    the inclusion  
    $\twoset{x\in\Omega_U}{ \dual B\popt_0(x)=0 }\subset A^c$
    for the complement $A^c=\Omega_U\backslash A$
    of $A$ and in addition    
   \begin{enumerate}
       \item (source condition) 
          \begin{equation}\label{E:source}
               \chi_{A^c} \uopt_0 = \chi_{A^c} P_{\Uad}(\dual Tw).
          \end{equation}
       \item (($\popt_0$-)measure condition)
          \begin{equation}\label{E:struct}
              \forall\ \epsilon > 0:\quad
               \meas(\twoset{x\in A}{0 \le 
                    \abs{\dual B\popt_0(x)} \le \epsilon})
                        \le C\epsilon^\kappa
          \end{equation}
         with the convention that $\kappa :=\infty$ if the left-hand
         side of \eqref{E:struct} is zero for some $\epsilon > 0$.
   \end{enumerate} 
\end{assu}
For a discussion of this assumption we refer to the texts subsequent to 
\cite[Assumption~7]{daniels} or \cite[Assumption~15]{dissnvd}.

Key ingredient in the analysis of the regularization error and
also of the discretization error considered later is the following lemma,
see \cite[Lemma~8]{daniels} or \cite[Lemma~16]{dissnvd} for a proof.

\begin{lemm}\label{L:l1reg} 
  Let Assumption~\ref{A:sourcestruct}.2 hold.
  For the solution $\uopt_0$ of \eqref{OCPl}, there holds 
  with some constant $C>0$ independent of $\alpha$ and $u$ 
  \begin{equation}\label{E:l1reg}
       C\norm{u-\uopt_0}_{L^1(A)}^{1+1/\kappa} 
            \le (\dual B\popt_0,u-\uopt_0)_U\quad\forall\ u\in\Uad.
  \end{equation}
\end{lemm}

Using this Lemma, we can now state regularization error estimates. 
\begin{theo}\label{T:regmain}
  For the regularization error there holds with positive constants 
  $c$ and $C$ indepent of $\alpha > 0$ the following.
  \begin{enumerate}
  \item Let Assumption~\ref{A:sourcestruct}.2 be satisfied with 
        $\meas(A^c)=0$ (measure condition holds a.e. on the domain). 
        Then the estimates
     \begin{align}
       \norm{\uopt_{\alpha}-\uopt_0}_{L^1(\Omega_U)}
                              &\le C\alpha^{\kappa}
                 \label{E:regconvumeas} \\
       \norm{\uopt_{\alpha}-\uopt_0}_U &\le C\alpha^{\kappa/2}
                 \label{E:regconvumeas2} \\
       \norm{\yopt_{\alpha}-\yopt_0}_H &\le C\alpha^{(\kappa+1)/2}
                 \label{E:regconvymeas}
     \end{align} 
        hold true.
        If $\kappa > 1$ holds and in addition
        \begin{equation}\label{E:dualTlinfty}
            \dual T:\range(T)\to L^\infty(\Omega_U) 
            \quad\quad\text{exists and is continuous,}
        \end{equation}
        we can improve 
        \eqref{E:regconvymeas} to
       \begin{equation}\label{E:regconvymeasimp}
           \norm{\yopt_{\alpha}-\yopt_0}_H \le C\alpha^{\kappa}.
       \end{equation}
  \item Let Assumption~\ref{A:sourcestruct} be satisfied with 
        $\meas(A)\cdot\meas(A^c) > 0$
        (source and measure condition on parts of the domain). 
        Then the following estimates 
     \begin{align}
       \norm{\uopt_{\alpha}-\uopt_0}_{L^1(A)}
                    &\le C\alpha^{\min(\kappa,\,\frac 2{1+1/\kappa})}
		 \label{E:regconvusourcemeas} \\
       \norm{\uopt_{\alpha}-\uopt_0}_U  
                    &\le C\alpha^{\min(\kappa,\,1)/2}
 		\label{E:regconvusourcemeas2} \\
       \norm{\yopt_{\alpha}-\yopt_0}_H 
                    &\le C\alpha^{\min((\kappa+1)/2,\,1)}
 		\label{E:regconvysourcemeas}
     \end{align}
        hold true. 

        If furthermore $\kappa > 1$ and \eqref{E:dualTlinfty} hold,
        we have the improved estimate
       \begin{equation}\label{E:regconvusourcemeasimp}
          \norm{\uopt_{\alpha}-\uopt_0}_{L^1(A)}
                \le C\alpha^\kappa.
       \end{equation}
  \end{enumerate}
\end{theo}
For a proof of this recent result, we refer to 
\cite[Theorem~11]{daniels} and \cite[Theorem~19]{dissnvd},
where also a discussion can be found.
We only recall two points for convenience here:

The assumption of the first case of the above Theorem implies 
\begin{equation}\label{E:bangbangsol}
    \meas(\twoset{x\in\Omega_U}{ \dual B\popt_0(x)=0 })=0,
\end{equation}
which induces bang-bang controls, compare Remark~\ref{R:bangbang}.

By Lemma~\ref{L:regOCPposalpha} and Remark~\ref{R:regOCPalpha0}
we can immediately see that the assumption \eqref{E:dualTlinfty}
on $\dual T$ is fulfilled for our parabolic problem. 

\subsection{Bang-bang controls}
We now introduce a second measure condition which leads to 
an improved bound on the decay of smoothness in the derivative of the 
optimal control when $\alpha$ tends to zero. 
This bound will be useful later to derive
improved convergence rates for the discretization errors.

\begin{defi}[$\popt_\alpha$-measure condition]
   If for the set
   \begin{equation}\label{E:setialpha}
     I_\alpha:=\twoset{x\in \Omega_U}
                 {\alpha a < -\dual B\popt_\alpha < \alpha b}
   \end{equation}
   the condition
   \begin{equation}\label{E:struct2}
     \exists\ \bar\alpha >0\ \forall\ 0<\alpha<\bar\alpha:
       \quad   \meas(I_\alpha)\le C\alpha^\kappa
   \end{equation}
   holds true 
   (with the convention that $\kappa := \infty$ if the measure in 
   \eqref{E:struct2} is zero for all $0<\alpha<\bar\alpha$),
   we say that the 
   \emph{$\popt_\alpha$-measure condition} is fulfilled. 
\end{defi}

\begin{theo}\label{T:convpalpha}
  Let us assume the \emph{$\sigma$-condition}
  \begin{equation}\label{E:sigmacond}
    \exists\ \sigma > 0\ \forall'\ x\in\Omega_U:\quad
          a \le -\sigma < 0 < \sigma \le b.
  \end{equation}

  If the $\popt_\alpha$-measure condition \eqref{E:struct2}
  is valid, Theorem~\ref{T:regmain}.1 holds, omitting its
  first sentence (``Let Assumption...'').
\end{theo}
\begin{proof}
  See \cite[Theorem~15]{daniels} or \cite[Theorem~24]{dissnvd}.
\end{proof}

If the limit problem is of certain regularity,
both measure conditions coincide: 
\begin{coro}
   Let a bang-bang solution be given, i.e., \eqref{E:bangbangsol} 
   holds true.
   In the case of $\kappa > 1$,
   \eqref{E:dualTlinfty}, and 
   the $\sigma$-condition \eqref{E:sigmacond}, 
   both measure conditions are equivalent.
\end{coro}
\begin{proof}
  See \cite[Corollary~18]{daniels} or \cite[Corollary~27]{dissnvd}.
\end{proof}

Let us now consider located controls.
Since $\popt_\alpha\in C^1(\bar I,\H)$ for $\alpha \ge 0$ by 
Lemma~\ref{L:regOCPposalpha} and Remark~\ref{R:regOCPalpha0},
we conclude
\begin{equation*}
   \norm{\partial_t \dual B\popt_\alpha}_{L^\infty(I)}
  \le C  \norm{\partial_t \popt_\alpha}_{L^\infty(I,\H)}
  \le C+C \norm{\uopt_\alpha}_U
  \le C
\end{equation*}
with a constant $C>0$ independent of $\alpha$ due to the definition
of $\Uad$. Recall that 
$a$, $b$ $\in$ $W^{1,\infty}(I)$ by
Assumption~\ref{A:Regularity}.
With this estimate, the projection formula \eqref{FONC}
and the stability of the projection (see \cite[Lemma~11]{dissnvd}) 
we obtain the bound
\begin{equation}\label{E:uaainf}
   \norm{\partial_t \uopt_\alpha}_{L^\infty(I)}
        \le \frac 1\alpha 
   \norm{\partial_t \dual B\popt_\alpha}_{L^\infty(I)}
      + \norm{\partial_t a}_{L^\infty(I)}
      + \norm{\partial_t b}_{L^\infty(I)}\\
    \le C\frac 1\alpha,
\end{equation}
if $\alpha > 0$ is sufficiently small.

If the $\popt_\alpha$-measure condition \eqref{E:struct2} is valid,
this decay of smoothness in terms of
$\alpha$ can be relaxed in weaker norms, as the following Lemma shows.

\begin{lemm}[Smoothness decay in the derivative]\label{L:smdecderiv}
   Let the  $\popt_\alpha$-measure condition \eqref{E:struct2}
   be fulfilled and located controls be given. Then 
   for $\alpha >0$ sufficiently small
   there holds
   \begin{equation}
    \norm{\partial_t \uopt_\alpha}_{L^p(I)}
     \le C\max(C_{ab},\alpha^{\kappa/p-1})
    \end{equation}
    with a constant $C>0$ independent of $\alpha$.
    Here, 
   $  
     C_{ab} := \norm{\partial_t a}_{L^\infty(I)}
      + \norm{\partial_t b}_{L^\infty(I)}
   $
   and $1\le p < \infty$.
    Note that $C_{ab}=0$ in the case of constant control bounds 
    $a$ and $b$.
\end{lemm}
\begin{proof}
  See \cite[Lemma~19]{daniels} or \cite[Lemma~28]{dissnvd}.
\end{proof}

The question of necessity of Assumption~\ref{A:sourcestruct}
and the $\popt_\alpha$-measure condition \eqref{E:setialpha}
to obtain the convergence rates of Theorem~\ref{T:regmain}.1
is discussed in
\cite[sections~4 and 5]{daniels} and
\cite[sections~1.4.3 and 1.4.4]{dissnvd}. 
The results there show 
that in several cases the conditions
are in fact necessary to obtain the 
convergence rates from above.

\section{The discretized problem}

\subsection{Discretization of the optimal control problem}

Consider a partition $0=t_0 <  t_1 < \dots < t_M=T$ of the time 
interval $\bar I$. With $I_m=[t_{m-1},t_m)$ we have 
$[0,T)=\bigcup_{m=1}^M I_m$.
Furthermore, let $t_m^*=\frac{t_{m-1}+t_m}{2}$ for $m=1,\dots,M$
denote the interval midpoints. By 
$0=:t_0^* < t_1^* < \dots < t_M^* < t_{M+1}^*:=T$ we get a 
second partition of $\bar I$, the so-called \emph{dual partition}, namely 
$[0,T)=\bigcup_{m=1}^{M+1} I_m^*$, with $I_m^*=[t_{m-1}^*,t_m^*)$.
The grid width of the first mentioned (primal) partition is given by the 
parameters $k_m=t_m-t_{m-1}$ and
 \[ k=\max_{1\le m\le M} k_m. \]
Here and in what follows we assume $k<1$.
We also denote by $k$ (in a slight abuse of notation) the grid itself.

We need the following conditions on sequences of time grids.
\begin{assu}
  There exist constants $0 < \kappa_1\le\kappa_2 < \infty$ 
  and $\mu >0$
  independent of $k$ such that there holds
  \[ 
     \forall\ m\in\{1,2,\dots,M-1\}:
           \kappa_1 \le \frac{k_m}{k_{m+1}} \le \kappa_2 
      \quad\text{and}\quad
      k \le \mu \min_{m=1,2,\dots,M} k_m. 
  \]
\end{assu}

On these partitions of the time interval, we define the Ansatz and
test spaces of the Petrov--Galerkin schemes. 
These schemes will replace the continuous-in-time weak 
formulations of the state equation and the adjoint
equation, i.e., \eqref{E:WFM} and \eqref{E:AA}, respectively.
To this end, we define at first for an arbitrary Banach space $X$
the semidiscrete function spaces
\begin{subequations}\label{e:semfs}
\begin{align}
P_k(X):&=\twoset{v\in C([0,T],X)}{\restr{v}{I_m}\in \mathcal P_1(I_m,X)}
\hookrightarrow H^1(I,X),\\  
P_k^*(X):&=\twoset{v\in C([0,T],X)}{\restr{v}{I_m^*}\in 
\mathcal P_1(I_m^*,X)}
\hookrightarrow H^1(I,X),\\  
\shortintertext{and}
Y_k(X):&=\twoset{v:[0,T]\rightarrow \dual{X}}{\restr{v}{I_m}\in 
\mathcal P_0(I_m,X)}\,.  
\end{align}
\end{subequations}
Here, $\mathcal P_i(J,X)$, $J\subset \bar I$, $i\in\{0,1\}$, is
the set of polynomial functions in time with degree of at most $i$ on the
interval $J$ with values in $X$.
We note that functions in $P_k(X)$ can be uniquely determined by 
$M+1$ elements from $X$. The same holds true for functions $v\in Y_k(X)$
but with $v(T)$ only uniquely determined in $\dual X$ by definition of
the space. The reason for this is given in the discussion below 
\cite[(2.16), p. 41]{dissnvd}.
Furthermore, for each function $v\in Y_k(X)$ we have $[v]\in L^2(I,X)$
where $[.]$ denotes the equivalence class with respect to
the almost-everywhere relation.

In the sequel, we will frequently use the following 
interpolation operators.
\begin{enumerate}
\item (Orthogonal projection) 
     $\mathcal P_{Y_k(X)}:L^2(I,X)\rightarrow Y_k(X)$
\begin{equation}\label{D:oproj}
 \restr{\mathcal  P_{Y_k(X)} v}{I_m}:=\frac{1}{k_m}\int_{t_{m-1}}^{t_m}
 v\, dt,\ m=1,\dots,M, \ \mathcal P_{Y_k(X)}v(T):=0
\end{equation}
\item (Piecewise linear interpolation on the dual grid)\\ 
$\pi_{P_k^*(X)} : C([0,T],X)\cup Y_k(X)\rightarrow P_k^*(X)$
\begin{equation}\label{D:lproj}
\begin{aligned}
\restr{\pi_{P_k^*(X)} v }{I_1^*\cup I_2^*} &:= v(t_1^*) + 
\frac{t-t_1^*}{t_2^*-t_1^*} \left(v(t_2^*)-v(t_1^*)\right),\\
\restr{\pi_{P_k^*(X)} v }{I_m^*} &:= v(t_{m-1}^*) + 
\frac{t-t_{m-1}^*}{t_m^*-t_{m-1}^*} \left(v(t_m^*)-v(t_{m-1}^*)
\right)\\
 &\quad\quad\text{ for } m=3,\dots,M-1,\\
\restr{\pi_{P_k^*(X)} v }{I_M^*\cup I_{M+1}^*} &:= v(t_{M-1}^*) + 
\frac{t-t_{M-1}^*}{t_M^*-t_{M-1}^*} \left(v(t_M^*)-v(t_{M-1}^*)
\right).
\end{aligned}
\end{equation}
\end{enumerate}

The interpolation operators are obviously linear mappings. 
Furthermore, they are bounded, and we have error estimates, as 
\cite[Lemma~31]{dissnvd} shows.

In addition to the notation introduced after Remark~\ref{R:bangbang},
adding a subscript $I_m$ to a norm will indicate an 
$L^2(I_m,\H)$ norm in the following. 
Inner products are treated in the same way.

Note that in all of the following results $C$ denotes a generic, strict 
positive real constant that does not depend on quantities which appear
to the right or below of it.

Note that we can extend the bilinear form $A$ of
\eqref{bilinA} in its first argument to $W(I)\cup Y_k(\V)$, thus
consider the operator
\begin{equation}\label{bilinAe}
A:W(I)\cup Y_k(\V) \times W(I) \rightarrow \mathbb{R},\quad
    \text{$A$ given by \eqref{bilinA}}. 
\end{equation}

Using continuous piecewise linear functions in space, we can
formulate fully discretized variants of the state and adjoint
equation.

We consider a regular triangulation $\mathcal T_h$ of $\Omega$
with mesh size 
\[
    h:=\max_{T\in\mathcal T_h}\diam(T),
\]
see, e.g., \cite[Definition (4.4.13)]{brenner-scott}, and
$N=N(h)$ triangles.
We assume that $h < 1$.
We also denote by $h$ (in a slight abuse of notation) the grid itself.

With the space
\begin{equation}
    X_h := \twoset{\phi_h\in C^0(\bar\Omega)}{\restr{\phi_h}{T}\in
                 \mathcal P_1(T,\mathbb R)
                     \quad\forall\ T\in\mathcal T_h}
\end{equation}
we define $X_{h0} := X_h \cap \V$ to discretize $\V$. 

For the space grid we make use of a standard
grid assumption, as we did for the time grid,
sometimes called quasi-uniformity.

\begin{assu}\label{A:grid2h}
  There exists a constant $\mu > 0$ independent of $h$ such that
  \[  h \le \mu \min_{T\in\mathcal T_h}\diam(T). \]
\end{assu}

We fix fully discrete ansatz and test spaces,
derived from their semidiscrete counterparts from \eqref{e:semfs},
namely
\begin{equation}\label{E:spaceskh}
    P_{kh}:=P_k(X_{h0}),\quad P_{kh}^*:=P_{kh}^*(X_{h0}),\quad 
   \text{and}\ Y_{kh}:=Y_k(X_{h0}).
\end{equation}
With these spaces, we introduce 
fully discrete state and adjoint equations as follows.
\begin{defi}[Fully discrete adjoint equation]\label{D:pkh}
For $h\in \LIIVd$ find $p_{kh}\in P_{kh}$ such that
\begin{equation}\label{E:AdjDiscrh}
A(\tilde y,p_{kh})
=\int_0^T \langle h(t),\tilde y(t)\rangle_{\Vd\V}\, dt
\quad\forall\ \tilde y \in Y_{kh}.
\end{equation}
\end{defi}
\begin{defi}[Fully discrete state equation]\label{D:ykh}
For $(f,g)\in L^2(I,\Vd)\times \H$ find $y_{kh}\in Y_{kh}$, such that
\begin{equation}\label{E:WFDh}
A(y_{kh},v_{kh})= \int_0^T \langle f(t),v_{kh}(t)\rangle_{\Vd\V}\, dt +
    (g,v_{kh}(0)) \quad\forall\ v_{kh}\in P_{kh}.
\end{equation}
\end{defi}
Existence and uniqueness of these two schemes follow as in the 
semidiscrete case discussed in \cite{DanielsHinzeVierling} or
\cite[section~2.1.2]{dissnvd}.

Let us recall some stability results and error estimates for these
schemes. The first result is \cite[Lemma~56]{dissnvd}.

\begin{lemm}\label{L:AdjDiscrStabh}
 Let $p_{kh}\in P_{kh}$ solve \eqref{E:AdjDiscrh} with $h\in \LIIH$.
 Then there exists a constant $C>0$ independent of $k$ and $h$ such that
  \[
      \norm{p_{kh}}_{H^1(I,\H)} + \norm{\nabla p_{kh}}_{C(\bar I,\H)}
       \le C\norm{h}_I.
  \]
\end{lemm}

For stability of a fully discrete state 
$y_{kh}$ and an error estimate, we recall \cite[Lemma~59]{dissnvd}.
 
\begin{lemm}\label{L:yDiscrStabh}
Let $y$ be the solution of \eqref{E:WFM} for some 
$(f,g) \in \LIIVd\allowbreak\times \H$ and let 
$y_{kh}\in Y_{kh}$ be the solution of
\eqref{E:WFDh} for the same $(f,g)$.
Then with a constant $C>0$ independent of $k$ and $h$, it holds
\[ \norm{y_{kh}}_I \le C
           \left(\norm{f}_{\LIIVd}+\norm{g}\right). 
\]
If furthermore the regularity 
$f\in\LIIH$ as well as $g\in\V$ is fulfilled,
we have the error estimate
\begin{equation}\label{E:ykherr}
   \norm{y-y_{kh}}_I
       \le C(h^2+k)\left( \norm{f}_I + \norm{\nabla g} \right).
\end{equation}
\end{lemm}

Let us now consider the error of the fully discrete adjoint state.
We begin with an $\LIIH$ norm result, which is \cite[Lemma~62]{dissnvd}.

\begin{lemm}\label{L:pkherr}
Let $p$ solve \eqref{E:AA} for some $h$ such that
$p$ has the regularity
$p\in H^1\left(I,H^2(\Omega)\cap\V\right)\bigcap H^2\left(I,\H\right)$.
Let furthermore $p_{kh}\in P_{kh}$ solve \eqref{E:AdjDiscrh} 
for the same $h$. Then it holds
\begin{equation*}
\norm{p_{kh}-p}_I
   \le C (k^2+h^2) (\norm{p_{tt}}_I +
                      \norm{\Delta p_t}_I).
\end{equation*}
\end{lemm}

For the pointwise-in-time error, we recall \cite[Lemma~65]{dissnvd}:
\begin{lemm}\label{L:errphklool2}
   Let the assumptions of Lemma~\ref{L:pkherr} be fulfilled.
   Then it holds
   \begin{equation*}
       \norm{p-p_{kh}}_{L^\infty(I,\H)}
       \le C(h^2+k) \left(\norm{\Delta p_t}_I
      +\norm{p_t}_{L^\infty(I,\H)}   \right).
  \end{equation*}
If in addition $p\in H^2(I,H^2(\Omega)\cap\V)$
and $p_{tt}\in L^\infty(I,\H)$ is known to hold,
we have the improved estimate
   \begin{multline*}
       \norm{p-p_{kh}}_{L^\infty(I,\H)}
  \le C(h^2+k^2)
    \left( \norm{\Delta p_t}_I
        +\norm{p_{t}}_{L^\infty(I,\H)} \right)  \\
       + Ck^2 \left ( \norm{\Delta p_{tt}}_I
              +\norm{p_{tt}}_{L^\infty(I,\H)}   \right).
  \end{multline*}
\end{lemm}

The following superconvergence result, which is \cite[Lemma~66]{dissnvd},
will also be used in the later error analysis.

\begin{lemm}\label{L:ImprRatekh} 
   Let $y\in Y$ and $y_{kh}\in Y_{kh}$ solve 
   \eqref{E:WFM} and \eqref{E:WFDh}, respectively,
   with data $(f,g)$ fulfilling
   $f\in H^1(I,\H)$, $f(0)\in \V$, $g\in\V$, and $\Delta g\in\V$. 
   By $p_{kh}(h)\in P_{kh}$ we denote the 
   solution to \eqref{E:AdjDiscrh} with right-hand side $h$. 
   Then it holds
  \begin{equation*} 
    \norm{y_{kh}-\mathcal P_{Y_k}y}_I 
     +\norm{p_{kh}(y_{kh}-y)}_{C(\bar I,\H)}
           \le C (k^2 F_1(f,g)+h^2 F_2(f,g))
   \end{equation*}
   with \quad\quad
      $   F_2(f,g):=\norm{f}_I +\norm{g}_{H^1(\Omega)} $\\
   and \quad\quad
      $   F_1(f,g):=F_2(f,g)
           +\norm{\partial_t f}_I
           +\norm{f(0)}_{H^1(\Omega)}
           +\norm{\Delta g}_{H^1(\Omega)}$.
\end{lemm}

We are now able to introduce
the discretized optimal control problem which reads

\begin{equation}\label{OCPkh}\tag{$\mathbb P_{kh}$}
\begin{aligned}
&\min_{ y_{kh}\in Y_{kh},u\in\Uad} 
     J(y_{kh},u)=\min \frac{1}{2}\norm{y_{kh}-y_d}^2_I+
             \frac{\alpha}{2}\norm{u}^2_U,\\
&\text{s.t. } y_{kh}=S_{kh}(Bu,y_0)
\end{aligned}
\end{equation}
where $\alpha$, $B$, $y_0$, $y_d$, and $\Uad$ are chosen as for
\eqref{OCP} and
$S_{kh}$ is the solution operator associated to the fully 
discrete state equation \eqref{E:WFDh}. 
Recall that the space $Y_{kh}$ was introduced in \eqref{E:spaceskh}.

For every $\alpha > 0$, this problem admits a unique solution triple
($\uoptd$, $\yoptd$, $\poptd$) where
$\yoptd = S_{kh}(B\uoptd,y_0)$ and 
$\poptd$ denotes the discrete adjoint state which is
the solution of the fully discrete adjoint equation 
\eqref{E:AdjDiscrh} with right-hand side $h:=\yoptd-y_d$.
The first order necessary and sufficient optimality condition for
problem \eqref{OCPkh} is given by
\begin{equation}\label{VarIneqkh}
   \uoptd\in\Uad,\quad
   \left( \alpha\uoptd + \dual B\poptd,u-\uoptd\right)_U 
   \ge 0\quad \forall\ u\in\Uad,
\end{equation}
which can be rewritten as
\begin{equation}\label{FONCkh}
    \uoptd = P_{\Uad}\left(-\frac{1}{\alpha}\dual B\poptd\right).
\end{equation}
The above mentioned facts can be proven in the same way
as for the continuous problem \eqref{OCP}. 

Note that the control space $U$ is not discretized in the
formulation \eqref{OCPkh}. In the numerical treatment, 
the relation \eqref{FONCkh} is instead exploited to get a discrete 
control.
This approach is called
\emph{Variational Discretization} and was introduced 
in \cite{Hinze2005}, see also \cite[Chapter 3.2.5]{hpuu}
for further details.

\begin{rema}\label{R:u0h}
In the case $\alpha = 0$, problem \eqref{OCPkh} has at least
one solution, but only $\yoptd$ and $\poptd$ are unique, whereas
an associated optimal control is in general non-unique. The reason
is that $f\mapsto S_{kh}(f,y_0)$ is not injective in contrast to 
$f\mapsto S(f,y_0)$. 
However, the discrete solution is unique (and of bang-bang type)
if the zero level set of $\dual B\poptd$ has measure zero.
\end{rema}

\subsection{Error estimates for the regularized 
                 problem}\label{sub:errestregprob}

In what follows, we use the notation $y_{kh}(v):= S_{kh}(Bv,y_0)$
with $v\in\Uad$, and $p_{kh}(h)$ is an abbreviation of the solution
to \eqref{E:AdjDiscrh} with right-hand side $h\in\LIIVd$.
Furthermore, $y(v)$ and $p(h)$ denote the continuous counterparts.
Note that therefore we have $\yopt=y(\uopt)$, $\yoptd = y_{kh}(\uoptd)$,
$\popt=p(\yopt-y_d)$, and $\poptd = p_{kh}(\yoptd-y_d)$.

The following Lemma provides a first step towards an error estimate
with respect to the control and state discretization. 
\begin{lemm}\label{L:errudu}
  Let $\uopt$ and $\uoptd$ solve \eqref{OCP} and \eqref{OCPkh},
  respectively, both for the same $\alpha \ge 0$. Then there holds
\begin{multline*}
\alpha  \norm{\uoptd-\uopt}_U^2 +\norm{\yoptd-y_{kh}(\uopt)}^2_I\\ 
  \le \left( 
   \dual B\Big( p_{kh}(\yopt-y_d)-\popt 
        + p_{kh}(y_{kh}(\uopt)-\yopt) \Big),
    \uopt-\uoptd\right)_U.
\end{multline*}
\end{lemm}
\begin{proof}
Inserting $\uoptd$ into \eqref{VarIneq} and $\uopt$ into 
\eqref{VarIneqkh} and adding up the resulting inequalities yields
\[
    \Big(\alpha (\uoptd-\uopt) + \dual B(\poptd-\popt),
         \uoptd-\uopt \Big)_U \le 0.
\]
After some simple manipulations we obtain
\begin{multline*}
\alpha  \norm{\uoptd-\uopt}^2_U
  \le \left( \dual B\Big(p_{kh}(\yopt-y_d)-\popt+ p_{kh}(y_{kh}(\uopt))
       -p_{kh}(\yopt)\Big), \uopt-\uoptd\right)_U\\
  +\left( \dual B\Big(\poptd- p_{kh}(y_{kh}(\uopt)-y_d)\Big),
             \uopt-\uoptd\right)_U,
\end{multline*}
and since the last line equals
$-\norm{\yoptd-y_{kh}(\uopt)}^2_I$, we end up with the desired estimate
by moving this term to the left.
\end{proof}

We can now prove an error estimate, which resembles the standard estimate
for variational discretized controls. It is build upon
\cite[Theorem~5.2]{DanielsHinzeVierling}. 
Since we are interested in
the limit behavior $\alpha\to 0$, we try to give a precise dependence
of the right-hand side on $\alpha$. Note the splitting in terms of
the quantities $d_0$ and $d_1$. 
In contrast to $d_0$, the term $d_1$ is \emph{not} bounded if
$\alpha\to 0$.

\begin{theo}\label{T:discrerr}
  Let $\uopt$ and $\uoptd$ solve \eqref{OCP} and \eqref{OCPkh},
  respectively, both for the same $\alpha \ge 0$.
  Then there exists
  a constant $\alpha_{\max} > 0$ independent of $k$ and $h$, 
  so that for all $0\le \alpha \le \alpha_{\max}$
  (with the convention ``$1/0=\infty=d_1$'' in the case of $\alpha=0$)
  the estimate
\begin{equation}\label{E:erruopt}
\begin{aligned}
\sqrt\alpha  \norm{\uoptd-\uopt}_U &+\norm{\yoptd-y_{kh}(\uopt)}_I\\ 
&\le 
    C \min\left(\frac{k^2+h^2}{\sqrt\alpha}d_0, 
             (k+h)\sqrt{\norm{\uoptd-\uopt}_U}\sqrt{d_0} \right) \\
  &\quad\quad + C\min\left(k^2d_1,kd_0\right)+Ch^2d_0 \\
&\le 
    C \max(d_0+1,\sqrt{d_0})
      \min\left(\frac{k^2}{\alpha}+\frac{h^2}{\sqrt\alpha},k+h \right)
\end{aligned}
\end{equation}
is satisfied with the constants $d_0 = d_0(\uopt)$ and 
$d_1 = d_1(\uopt)$
from the estimates \eqref{E:uypapriori} in Lemma~\ref{L:regOCPposalpha}.
\end{theo}
\begin{proof}
We split the right-hand side of the estimate 
from Lemma~\ref{L:errudu} and get with the Cauchy-Schwarz inequality
\begin{multline*}
 \alpha  \norm{\uoptd-\uopt}^2_U + \norm{\yoptd-y_{kh}(\uopt)}^2_I\\ 
   \le  
     \norm { p_{kh}(\yopt-y_d)-\popt }_I \norm{\uopt-\uoptd}_U
       + \left(\dual B\left(
                 p_{kh}\left(y_{kh}\left(\uopt\right)-\yopt\right)
                     \right), \uopt-\uoptd\right)_U 
                   = I+II.
\end{multline*}
With the help of Lemma~\ref{L:pkherr} and Lemma~\ref{L:regOCPposalpha},
we conclude 
\begin{equation*}
  \norm{p_{kh}(\yopt-y_d)-\popt }_I
 \le C(k^2+h^2) ( \norm{\popt_{tt}}_I
          + \norm{\Delta \popt_t}_I) \le C(k^2+h^2)d_0.
\end{equation*}
Now we use Cauchy's inequality to obtain 
\begin{equation*}
  I \le \frac{C}{\alpha}
          \norm{  p_{kh}(\yopt-y_d)-\popt }_I^2
       +\frac{\alpha}{2} \norm{ \uopt-\uoptd }_U^2.
\end{equation*}
Here, the second addend can be moved to the left.
Both estimates can be summarized as
\begin{equation*}
  \sqrt I  \le C \min\left(\frac{k^2+h^2}{\sqrt\alpha}d_0, 
             (k+h)\sqrt{\norm{\uoptd-\uopt}_U}\sqrt{d_0}\right).
\end{equation*}
The addend $II$ can be estimated as
\begin{equation*}
  II = (y_{kh}(\uopt)-\yopt,
        y_{kh}(\uopt)-\yoptd)_I
     \le \frac 12 (\norm{y_{kh}(\uopt)-\yopt}_I^2
         + \norm{y_{kh}(\uopt)-\yoptd}_I^2 ).
\end{equation*}
We move the second term to the left.
Note that in the previous estimate $\yopt$ can be replaced
by $\mathcal P_{Y_k}\yopt$ by definition of $\mathcal P_{Y_k}$.
We thus can invoke either the error estimate of the state equation
\eqref{E:ykherr} from Lemma~\ref{L:yDiscrStabh} or the 
superconvergence result from Lemma~\ref{L:ImprRatekh}.
In conclusion, we have
\begin{equation*}
  \sqrt{II} \le C\min\left((k+h^2)d_0,k^2d_1+h^2d_0\right)
                  =\min\left(kd_0,k^2d_1\right)+h^2d_0. 
\end{equation*}
Together with the estimate for $\sqrt{I}$, we obtain 
the first inequality of the claim.

For the second inequality, we first note that
with the help of the projection
formula \eqref{FONC}, the stability of the projection, see, e.g., 
\cite[Lemma~11]{dissnvd}, and the regularity result 
\cite[Lemma~6]{dissnvd}
one immediately derives the estimate
\begin{equation}\label{E:highuopt}
\begin{aligned}
    &\norm{\uopt}_{H^1(I,\tilde U)} + \norm{B\uopt(0)}_{\HI}\\
          &\le \frac C\alpha (
    \norm{\popt}_{H^1(I,\H)} + \norm{\popt(0)}_{\HI}  )
        +C(a)+C(b)\\
   &\le \frac C\alpha ( \norm{y_d}_I +\norm{\uopt}_U+\norm{y_0}_{\HI})
        +C(a)+C(b) 
\end{aligned}
\end{equation}
where 
$\tilde U\in \{ \mathbb R, \H \}$, depending on whether located or 
distributed controls are given,
and
$
   C(x)= \norm{x}_{H^1(I,\tilde U)} + \norm{x(0)}_{X}     
$
with $X=\HI$ (distributed controls) or $X=\mathbb R$ (located controls).
This term is bounded due to Assumption~\ref{A:Regularity}.

Since there exists an $\alpha_{\max}>0$, depending only on the data
$a$, $b$, $y_0$, $y_d$, such that
\begin{equation}\label{E:uoptregdecay}
 \forall\ 0\le \alpha \le \alpha_{\max}:\quad\quad
          d_1 + d_1^+  \le C \frac 1\alpha (d_0+1)
\end{equation}
holds
with $d_1^+ := d_1^+(\uopt)$
from the estimates \eqref{E:uypapriori} in Lemma~\ref{L:regOCPposalpha},
and since $\sqrt{\norm{\uoptd-\uopt}_U}$ is bounded 
independently of $\alpha$ due to the definition of $\Uad$,
we get the claim.
\end{proof}
 From the proof of the previous Theorem, one can immediately
 derive a first robust (with respect to $\alpha\to 0$)
 error bound for the optimal state.
\begin{coro}\label{C:stateerrrob1}
  Let $\uopt$ and $\uoptd$ solve \eqref{OCP} and \eqref{OCPkh},
  respectively, both for the same arbitrarily chosen $\alpha \ge 0$.
  Then there holds
  \begin{equation*}
       \norm{\yopt-\yoptd}_I \le C (k + h)\max(d_0+1,\sqrt{d_0})
  \end{equation*}
  with a constant $C>0$ independent of $\alpha$ where
  $d_0$ is given in Theorem~\ref{T:discrerr}.
\end{coro}
\begin{proof}
  Combining
    \begin{equation*}
      \norm{\yopt-\yoptd}_I 
           \le  \norm{y_{kh}(\uopt)-\yoptd}_I
                   +\norm{\yopt-y_{kh}(\uopt)}_I
   \end{equation*}
  with the previous Theorem and 
  \eqref{E:ykherr} from Lemma~\ref{L:yDiscrStabh} proves the claim.
\end{proof}
Now, from the above Theorem we derive further non-robust estimates
for the discrete state and adjoint state.
Finally, we prove second order convergence for $\pi_{P_k^*}\yoptd$, i.e.,
the piecewise linear interpolation on the dual grid of the optimal state. 
This function is obtained for free from $\yoptd$,
since $\yoptd$ only has to be evaluated on the dual time grid.
Compare \cite[Theorem~5.3]{DanielsHinzeVierling} for the convergence
of the interpolation in the semidiscrete case.

\begin{coro}\label{C:nonrobest}
Let $\uopt$ and $\uoptd$ denote the solutions to \eqref{OCP}
and \eqref{OCPkh}, respectively, both for the same sufficiently
small $\alpha >0$ (in the sense of Theorem~\ref{T:discrerr}).
With $d_0$ and $d_1$ as in Theorem~\ref{T:discrerr} and 
\begin{equation*}
     d_1^+ := d_1^+(\uopt) = 
  d_1(\uopt) + C (\norm{\partial^2_t y_d}_I
   +\norm{\nabla \partial_t y_d(T)} 
   + \norm{\nabla \Delta y_d(T)}  
   +\norm{\nabla B\uopt(T)}
\end{equation*}
from the estimates \eqref{E:uypapriori} in Lemma~\ref{L:regOCPposalpha},
the estimates
\begin{equation*}
\begin{aligned}
  \norm{\uopt-\uoptd}_U 
       &\le C(\frac{k^2d_1}{\sqrt\alpha} + \frac{k^2+h^2}{\alpha}d_0)
       \le C(\frac{k^2}{\alpha^{3/2}}+\frac{h^2}{\alpha})(d_0+1),\\
  \norm{\yopt-\yoptd}_I
       &\le C(k+ \frac{k^2}{\alpha} + \frac{h^2}{\sqrt\alpha} )(d_0+1),
             \quad\text{and}\\
  \alpha \norm{\uopt-\uoptd}_{L^\infty(I,\tilde U)}
     &+ \norm{\popt-\poptd}_{L^\infty(I,\H)}  
        + \norm{\yopt-\pi_{P_k^*}\yoptd}_I\\
    &\le C(k^2d_1^+ + \frac{k^2+h^2}{\sqrt\alpha}d_0)
     \le C(  \frac{k^2}{\alpha}+\frac{h^2}{\sqrt\alpha})(d_0+1)
\end{aligned}
\end{equation*}
hold 
with $\tilde U\in \{ \mathbb R, \H \}$ depending on whether located or 
distributed controls are given.
\end{coro}
\begin{proof} 
The first estimate for the optimal control and the estimate for 
the optimal state follow from Theorem~\ref{T:discrerr}. For the
latter, we argue as in the proof of Corollary~\ref{C:stateerrrob1}.

For the optimal adjoint state, we split the error into three parts
to obtain with $L:=L^\infty(I,\H)$
  \begin{multline*}
  \norm{\popt-\poptd}_L\\
    \le \norm{\popt-p_{kh}(\yopt-y_d)}_L
        +\norm{p_{kh} (\mathcal P_{Y_k}\yopt -y_{kh}(\uopt))}_L
         + \norm{p_{kh} (y_{kh}(\uopt)- \yoptd)}_L.
  \end{multline*}
With the second error estimate from Lemma~\ref{L:errphklool2}, 
the regularity given in Lemma~\ref{L:regOCPposalpha}, and the estimate
from Lemma~\ref{L:ImprRatekh}, we conclude
\begin{equation*}
    \norm{\popt-p_{kh}(\yopt-y_d)}_L 
      + \norm{p_{kh} (\mathcal P_{Y_k}\yopt -y_{kh}(\uopt))}_L
           \le C (h^2d_0+k^2d_1^+),
\end{equation*}
since $d_1\le d_1^+$.

Stability from Lemma~\ref{L:AdjDiscrStabh} combined with
Theorem~\ref{T:discrerr} gives the estimate
\begin{equation*}
   \norm{p_{kh} (y_{kh}(\uopt)- \yoptd)}_L
     \le C \frac{k^2+h^2}{\sqrt\alpha}d_0 + Ck^2d_1 + Ch^2d_0. 
\end{equation*}
From this, we get
\begin{equation*}
  \norm{\popt-\poptd}_L
     \le C \frac{k^2+h^2}{\sqrt\alpha}d_0 + Ck^2d_1^+.
\end{equation*}

The projection formulae \eqref{FONC} and
\eqref{FONCkh}, Lipschitz continuity of the projection given in
\cite[Lemma~11]{dissnvd}, and stability of $\dual B$ yield
\begin{equation*}
  \norm{\uopt-\uoptd}_{L^\infty(I,\tilde U)} 
         \le C \frac 1\alpha \norm{\popt-\poptd}_L.
\end{equation*}
Together with the just established estimate
this yields the pointwise-in-time error estimate for the optimal control.

For the proof of the error $\norm{\yopt-\pi_{P_k^*}\yoptd}_I$, 
we refer the reader to \cite[Corollary~71]{dissnvd}.

Using the inequality \eqref{E:uoptregdecay}, we can finally reduce the
non-robust constants $d_1$ and $d_1^+$ to the robust one $d_0$.
\end{proof}
Let us comment on the estimates of 
Theorem~\ref{T:discrerr} and Corollary~\ref{C:nonrobest}.
These estimates show that if $\alpha > 0$ is fixed, we have
convergence rates $h^2+k^2$ except for the state error.
Invoking the regularization error,
one obtains estimates for the total error between the limit
problem and the discrete regularized one.
From this, a coupling rule for the parameters $\alpha$, $k$ and $h$
can be derived.

As an example, consider the error in the projected state
for the special case $\kappa = 1$. With the help of
Theorem~\ref{T:regmain},
and Corollary~\ref{C:nonrobest} we get with the inequality
\eqref{E:uoptregdecay} the estimate
\begin{multline}\label{E:totalprojstate}
   \norm{ \yopt_0 - \pi_{P_k^*}(\yoptd) }_I
   \le \norm{ \yopt_0 - \yopt_\alpha }_I
       + \norm{ \yopt_\alpha - \pi_{P_k^*}(\yoptd) }_I \\
   \le C(\alpha + k^2d_1^+ + \frac{k^2+h^2}{\sqrt\alpha}d_0)
   \le C (\alpha + \frac{k^2}{\alpha} + \frac{h^2}{\sqrt\alpha})
          (d_0+1),
\end{multline}
which implies
$\norm{ \yopt_0 - \pi_{P_k^*}(\yoptd) }_I \le C k = Ch^{4/3}$
when setting $\alpha = k = h^{4/3}$.

However, if the decay estimate $d_1^+ \le \frac C\alpha$,
i.e., \eqref{E:uoptregdecay}, can be
improved, we can get a better convergence rate (with respect to
$k$) for the total error. 
In Lemma~\ref{L:smdecderiv} we saw that 
this is indeed possible.

Unfortunately, space convergence of order $h^2$ is not achievable
in the above mentioned estimates if $\alpha$ tends to zero
due to $\alpha$ appearing in the denominator.
To overcome this, we establish other estimates in the next
subsection. The question of improving the decay estimate
\eqref{E:uoptregdecay} is discussed in the next but one
subsection using the estimates of the next subsection.

\subsection{Robust error estimates}

All the previous estimates (except Corollary~\ref{C:stateerrrob1})
are not robust for $\alpha\to 0$, since
$\alpha$ appears always in a denominator on the right-hand side.
Especially, convergence of order $h^2$ is not achievable as discussed
at the end of the previous subsection.  With some refined analysis, 
however, one can show estimates which are robust with respect 
to $\alpha\to 0$. A key ingredient is Lemma~\ref{L:l1reg},
which was also very important for the derivation of the regularization
error.

Recall the notation from the beginning of subsection
\ref{sub:errestregprob}.

\begin{theo}\label{T:fullrobust}
  Let Assumption~\ref{A:sourcestruct} be fulfilled so that
  either \eqref{E:regconvumeas} or \eqref{E:regconvusourcemeas} from
  Theorem~\ref{T:regmain} holds.
  We denote the valid convergence rate for the control
  by $\alpha^{\omega_1}$.
  Then, either \eqref{E:regconvymeas} 
  or \eqref{E:regconvysourcemeas} is fulfilled.
  We abbreviate the corresponding convergence rate by
  $\alpha^{\omega_2}$.

  Let $\uopt_0$ be the solution of \eqref{OCPl}
  with associated state $\yopt_0$. For some $\alpha \ge 0$ let
  in addition
  $\uopt_d:=\uopt_{\alpha,kh}\in\Uad$ be a 
  solution of \eqref{OCPkh} with associated 
  discrete state $\yopt_d$ and adjoint state $\popt_d$.
  Then there holds 
\begin{multline}\label{E:fullrobustu}
   \norm{\uopt_0-\uopt_d}_{L^1(A)} 
         \le C( \alpha^{\omega_1} \\ 
   + \norm{ \dual B(p_{kh}-p) ( y(\uopt_d)-y_d )}
              _{L^\infty(A)}^{\kappa} 
   + \norm{  \dual B(p_{kh}-p) ( y(\uopt_d)-y_d )}
              _{L^1(A^c)}^{\frac{1}{1+1/\kappa}}\\
   +\norm{ \dual Bp_{kh} (y_{kh}(\uopt_d) -y(\uopt_d))}
                   _{L^\infty(A)}^{\kappa} 
   +\norm{ \dual Bp_{kh} (y_{kh}(\uopt_d) -y(\uopt_d)) }
            _{L^1(A^c)}^\frac{1}{1+1/\kappa} )
\end{multline}
for the error in the control and
\begin{multline}\label{E:fullrobusty}
   \norm{\yopt_0-\yopt_d}_I
        \le C(\alpha^{\omega_2}  \\
   + \norm{ \dual B(p_{kh}-p)  ( y(\uopt_d)-y_d )}
              _{L^\infty(A)}^{\frac{1+\kappa}{2}}
   + \norm{ \dual B(p_{kh}-p)  ( y(\uopt_d)-y_d )}
              _{L^1(A^c)}^{1/2}  \\
   + \norm{ \dual Bp_{kh} (y_{kh}(\uopt_d) -y(\uopt_d))}
                   _{L^\infty(A)}^{\frac{1+\kappa}{2}}
   + \norm{ \dual Bp_{kh} (y_{kh}(\uopt_d) -y(\uopt_d)) }
            _{L^1(A^c)}^{1/2}\\
   + \norm{  y_{kh}(\uopt_d) -y(\uopt_d) }_I )
\end{multline}
for the error in the state.
\end{theo}
\begin{proof}
To the estimate \eqref{E:l1reg} from Lemma~\ref{L:l1reg} 
with $u:=\uopt_d$, i.e.,
\begin{equation}
       C\norm{\uopt_d-\uopt_0}_{L^1(A)}^{1+1/\kappa} 
            \le (-\dual B\popt_0,\uopt_0-\uopt_d)_U,
\end{equation}
we add the necessary
condition \eqref{VarIneqkh} for $\uopt_d$ with $u:=\uopt_0$, 
which can be rewritten as
\begin{equation}
   \alpha \norm{\uopt_0-\uopt_d}_U^2
   \le 
   \left( \alpha\uopt_0 + \dual B\popt_d,\uopt_0-\uopt_d\right)_U. 
\end{equation}
We end up with
\begin{multline}\label{E:fullrobust1}
   \norm{\uopt_0-\uopt_d}^{1+1/\kappa}_{L^1(A)}
     + \alpha \norm{\uopt_0-\uopt_d}_U^2
     + \norm{y(\uopt_0)-y(\uopt_d)}_I^2\\
   \le C\Big(
          -\dual Bp(y(\uopt_d)-y_d) + \dual Bp_{kh}(y_{kh}(\uopt_d)-y_d)
                            + \alpha\uopt_0,\uopt_0-\uopt_d
                \Big)_U\\
   \le C\Big(
          \underbrace{\dual B(p_{kh}-p)(y(\uopt_d)-y_d)}_{I}
          + \underbrace{ \dual Bp_{kh}(y_{kh}(\uopt_d) -y(\uopt_d)) }_{II}
       \\
          + \underbrace{\alpha \uopt_0}_{III},
          \uopt_0-\uopt_d    \Big)_U.
\end{multline}
We now use \cite[Lemma~18]{dissnvd}, Cauchy's and Young's inequality
to estimate $III$ as
\begin{multline*}
   \alpha (\uopt_0, \uopt_0-\uopt_d)_U
        \le \alpha C\left( \norm{T(\uopt_d-\uopt_0)}_H 
              + \norm{\uopt_d-\uopt_0}_{L^1(A)}\right)\\
        \le C\alpha^2 + \frac 14\norm{T(\uopt_d-\uopt_0)}_H^2
            +C\alpha^{1+\kappa} 
          + \frac 14 \norm{\uopt_d-\uopt_0}
                     _{L^1(A)}^{1+1/\kappa}.
\end{multline*}
The $\alpha$-free terms can now be moved to the left, since
$\norm{T(\uopt_d-\uopt_0)}_H = \norm{y(\uopt_d)-y(\uopt_0)}_I$.
Note that $C\alpha^2$ can be omitted if $A=\Omega_U$ since by
Young's inequality we then get
      \begin{equation*}
        \alpha \norm{\uopt_{\alpha}-\uopt_0}_{L^1(\Omega_U)}
           \le C\alpha^{\kappa+1} + C
              \norm{\uopt_{\alpha}-\uopt_0}_{
                    L^1(\Omega_U)}^{1+1/\kappa}.
      \end{equation*}
Thus only the term $C\alpha^{2\,\omega_2}$ remains
on the right-hand side.

For $I$ and $II$, we proceed with the help of Young's inequality to 
obtain
\begin{equation*}
\begin{aligned}
 &(\sim,\uopt_0-\uopt_d)_U \\
  &= (\sim,\uopt_0-\uopt_d)_{L^2(A)} 
     +(\sim,\uopt_0-\uopt_d)_{L^2(A^c)} \\
  &\le C\norm{\sim}^{1+\kappa}_{L^\infty(A)}
     + \frac 14\norm{\uopt_0-\uopt_d}^{1+1/\kappa}_{L^1(A)} 
     +\norm{\sim}_{L^1(A^c)}
                 \norm{b-a}_{L^\infty(A^c)}
\end{aligned}
\end{equation*}
and move the second addend to the left. 

Finally, we end up with
\begin{multline*}
   \norm{\uopt_0-\uopt_d}^{1+1/\kappa}_{L^1(A)}
     + \alpha \norm{\uopt_0-\uopt_d}_U^2
     + \norm{y(\uopt_0)-y(\uopt_d)}_I^2
   \le C( \alpha^{2\,\omega_2}  \\
   \phantom{\le}  + \norm{  \dual B(p_{kh}-p)(y(\uopt_d)-y_d) }
              _{L^\infty(A)}^{1+\kappa} 
        + \norm{  \dual B(p_{kh}-p) (y(\uopt_d)-y_d) }
              _{L^1(A^c)}\\
   \phantom{\le}  + \norm{ \dual Bp_{kh}(y_{kh}(\uopt_d) -y(\uopt_d))}
                   _{L^\infty(A)}^{1+\kappa} 
    + \norm{  \dual Bp_{kh}(y_{kh}(\uopt_d) -y(\uopt_d)) } 
                     _{L^1(A^c)}.
\end{multline*}
From this we conclude the claim for the optimal control.

The just established estimate together with the decomposition
\[
   \norm{\yopt_0-\yopt_d}_I \le
      \norm{  y_{kh}(\uopt_d) -y(\uopt_d) }_I
     +  \norm{y(\uopt_d)-y(\uopt_0)}_I
\]
yields the claim for the optimal state.
\end{proof}
\begin{rema}\label{R:02alpha}
The error estimate \eqref{E:fullrobustu} in the previous Theorem
for $\alpha > 0$ is also valid if $\uopt_0$ is replaced by 
$\uopt_\alpha$, i.e., the
solution of \eqref{OCP} for some $\alpha > 0$,
since by Theorem~\ref{T:regmain} we can estimate
\begin{equation}
\begin{aligned}
   \norm{\uopt_\alpha-\uopt_d}_{L^1(A)} 
   &\le 
   \norm{\uopt_\alpha-\uopt_0}_{L^1(A)} 
   +\norm{\uopt_0-\uopt_d}_{L^1(A)} \\
   &\le C\alpha^{\omega_1}
      + \norm{\uopt_0-\uopt_d}_{L^1(A)}.
\end{aligned}
\end{equation}
Likewise, in \eqref{E:fullrobusty} the state  
$\yopt_0$ can be replaced by $\yopt_\alpha$.

We will make use of this fact in the proof of the next Theorem.
\end{rema}

In combination with the error estimates for the state and
adjoint state equations previously derived, we
can now prove a first error estimate between solutions of
\eqref{OCPkh} and \eqref{OCPl}, which is robust if $\alpha$
tends to zero. In view of the numerical verification,
we restrict ourselves now to the situation $A=\Omega_U$ and
located controls.
\begin{theo}\label{T:fullrobust2}
  Let the assumptions of Theorem~\ref{T:fullrobust}
  be fulfilled. Further, we assume located controls and
  $A=\Omega_U$ (measure condition on the whole domain).

  Then there hold
  the estimates
  \begin{equation}\label{E:fullrobust2-1}
       \norm{\uopt_0-\uopt_d}_U^2 
     + \norm{\uopt_0-\uopt_d}_{L^1(\Omega_U)} 
       \le C \left( \alpha + h^2 + k \right)^\kappa  
            (1+d_0(\uopt_d)^\kappa)
  \end{equation}
  for the error in the control, for the auxiliary error
  \begin{equation}\label{E:fullrobust2-2}
    \norm{\yopt_d-y_{kh}(\uopt_\alpha)}^2_I 
    \le C 
    (h^2+k)d_0(\uopt_\alpha)
    \Big(
         \alpha^\kappa + (h^2+k)^\kappa d_0(\uopt_d)^\kappa 
    \Big)
  \end{equation}
  where by $\uopt_\alpha$ we denote the solution of \eqref{OCP},
  and
  \begin{equation}\label{E:fullrobust2-3}
   \norm{\yopt_d-\yopt_0}_I
        \le C \left( \alpha^{\frac{1+\kappa}{2}} 
                  + (h^2+k)^{\min(1,\frac{1+\kappa}{2})}
             \right)   
        \left( 1+d_0(\uopt_d)^{\min(1,\frac{1+\kappa}{2})} \right)
  \end{equation}
  for the error in the state. 

  If $\kappa >1$, we have the improved convergence rate
  \begin{equation}\label{E:fullrobust2-4}
   \norm{\yopt_d-\yopt_0}_I
    \le C(  \alpha^{\kappa} + h^2+k )
      (1+ \max \Big( d_0(\uopt_d)^\kappa, d_0(\uopt_\alpha) \Big) ),
  \end{equation}
  thus observe the regularization error \eqref{E:regconvymeasimp}.
\end{theo}
\begin{proof}
Combining Theorem~\ref{T:fullrobust} with 
the adjoint error estimate in Lemma~\ref{L:errphklool2},
the adjoint stability from Lemma~\ref{L:AdjDiscrStabh},
the error estimate \eqref{E:ykherr} in Lemma~\ref{L:yDiscrStabh}, 
and the regularity given in Lemma~\ref{L:regOCPposalpha} and
Remark~\ref{R:regOCPalpha0},
we achieve \eqref{E:fullrobust2-1} and \eqref{E:fullrobust2-3}
except for the $U$ error in the control.
This error can be derived from the corresponding $L^1$ error 
by the estimate
\begin{equation}\label{E:bangU21}
\begin{aligned}
  \norm{\uopt_0-\uopt_d}_U^2 
   &\le \norm{\uopt_0-\uopt_d}_{L^\infty(\Omega_U)}
       \norm{\uopt_0-\uopt_d}_{L^1(\Omega_U)}\\
   &\le \norm{b-a}_{L^\infty(\Omega_U)}
       \norm{\uopt_0-\uopt_d}_{L^1(\Omega_U)},
\end{aligned}
\end{equation}
which follows immediately from standard $L^p$ interpolation,
see, e.g., \cite[Theorem 2.11]{adams}, and the definition of $\Uad$.

Let us now tackle the improved state convergence, thereby proving
the estimate \eqref{E:fullrobust2-2}. We split the error into
three parts and obtain with the help of \eqref{E:regconvymeasimp} 
and the error estimate \eqref{E:ykherr} from Lemma~\ref{L:yDiscrStabh} 
\begin{equation*}
\begin{aligned}
 &\norm{\yopt_d-\yopt_0}_I^2\\
 &\le C\Big( \norm{\yopt_d-y_{kh}(\uopt_\alpha)}_I^2
    +\norm{y_{kh}(\uopt_\alpha)-y(\uopt_\alpha) }_I^2
    +\norm{y(\uopt_\alpha)-y(\uopt_0)}_I^2
      \Big) \\
 &\le C\Big( 
     \norm{\yopt_d-y_{kh}(\uopt_\alpha)}_I^2
    + (h^2+k)^2 d_0^2(\uopt_\alpha)
    + \alpha^{2\kappa} \Big),
\end{aligned}
\end{equation*}
where we also used \eqref{E:uypapriori} from Lemma~\ref{L:regOCPposalpha}.

For the remaining term, we invoke Lemma~\ref{L:errudu} in combination
with \eqref{E:fullrobustu} and Remark~\ref{R:02alpha}
and setting $L:=L^\infty(I,\H)$ we obtain with the stability of $\dual B$ 
for located controls
\begin{equation}\label{E:supconvykhualpha}
\begin{aligned}
&\norm{\yopt_d-y_{kh}(\uopt_\alpha)}^2_I\\ 
  &\le C 
    \Big(
     \norm{ p_{kh}(\yopt_\alpha-y_d)-\popt_\alpha }_L
    + \norm{ p_{kh}(y_{kh}(\uopt_\alpha)-\yopt_\alpha) }_L
    \Big)
    \norm{ \uopt_\alpha-\uopt_d }_{L^1(\Omega_U)}\\
  &\le C
    \Big(
     \norm{ p_{kh}(\yopt_\alpha-y_d)-\popt_\alpha }_L
    + \norm{ p_{kh}(y_{kh}(\uopt_\alpha)-\yopt_\alpha) }_L
    \Big)\cdot\\
    &\phantom{\le}\quad\quad
    \Big( \alpha^\kappa
       +\norm{ (p_{kh}-p) ( y(\uopt_d)-y_d )}_{L}^{\kappa} 
       +\norm{ p_{kh} (y_{kh}(\uopt_d) -y(\uopt_d))}_{L}^{\kappa}
    \Big). 
\end{aligned}
\end{equation}
Invoking again Lemma~\ref{L:errphklool2}, Lemma~\ref{L:AdjDiscrStabh},
estimate \eqref{E:ykherr} from Lemma~\ref{L:yDiscrStabh}, 
and Lemma~\ref{L:regOCPposalpha}, we get
\begin{equation*}
\norm{\yopt_d-y_{kh}(\uopt_\alpha)}^2_I 
\le C 
    (h^2+k)d_0(\uopt_\alpha)
    \Big( \alpha^\kappa + (h^2+k)^\kappa d_0^\kappa(\uopt_d) 
    \Big),
\end{equation*}
which is the auxiliary estimate \eqref{E:fullrobust2-2} of the statement.

If $\kappa > 1$, we can use the Cauchy-Schwarz inequality to get from it
the estimate
\begin{equation*}
\norm{\yopt_d-y_{kh}(\uopt_\alpha)}^2_I 
\le C \Big(
    (h^2+k)^2d_0^2(\uopt_\alpha)
     + \alpha^{2\kappa}
     + (h^2+k)^{1+\kappa} d_0(\uopt_\alpha) d_0^\kappa(\uopt_d)
     \Big).
\end{equation*}
Since $\kappa > 1$, collecting all estimates yields the inequality
  \begin{equation*}
   \norm{\yopt_d-\yopt_0}_I^2
    \le C(  \alpha^{2\kappa} + (h^2+k)^2 
      \max \Big( d_0^{2\kappa}(\uopt_d) , d_0^2(\uopt_\alpha) \Big) ),
  \end{equation*}
from which we finally get \eqref{E:fullrobust2-4}.
\end{proof}

\begin{coro}
 Let the assumptions of the previous theorem hold.
 For the adjoint state we have the error estimate
\begin{equation*}
  \norm{\popt_0-\popt_d}_{L^\infty(I,\H)}
    \le C( \alpha^{\max(\frac{1+\kappa}{2}, \kappa) } +  
      (k+ h^2)^{\min(1,\frac{1+\kappa}{2})}C(\uopt_d,\uopt_\alpha))
\end{equation*}
with $ C(\uopt_d,\uopt_\alpha) =
   \max(1,d_0(\uopt_d) , d_0(\uopt_\alpha))
   ^{\max(1,\frac{1+\kappa}{2})}$.
\end{coro}
\begin{proof}
Inspecting the proof of Corollary~\ref{C:nonrobest}, we get the
estimate
\begin{equation*}
  \norm{\popt_\alpha-\popt_d}_{L^\infty(I,\H)}
    \le C((k+ h^2)d_0(\uopt_\alpha)
            + \norm{y_{kh}(\uopt_\alpha)-\yopt_d}_I).
\end{equation*}
The last addend can be estimated with the auxiliary estimate
\eqref{E:fullrobust2-2} from the previous theorem and Cauchy's
inequality. We obtain
\begin{equation*}
     \norm{\popt_\alpha-\popt_d}_{L^\infty(I,\H)}
    \le C( \alpha^{\max(\frac{1+\kappa}{2}, \kappa) } +  
      (k+ h^2)^{\min(1,\frac{1+\kappa}{2})}C(\uopt_d,\uopt_\alpha)).
\end{equation*}
Invoking the regularization errors \eqref{E:regconvymeas} and
\eqref{E:regconvymeasimp} proves the claim.
\end{proof}

\subsection{Improved estimates for bang-bang controls}

As motivated at the end of subsection \ref{sub:errestregprob},
improving the decay estimate \eqref{E:uoptregdecay} 
with the help of Lemma~\ref{L:smdecderiv} 
leads to improved (non-robust) error estimates. 
However, the convergence rate $h^2$ is not achievable
in these estimates, but the robust estimates from 
Theorem~\ref{T:fullrobust} overcome this problem. 
On the other hand, in Theorem~\ref{T:fullrobust} we have $\uopt_d$
on the right-hand side instead of $\uopt_\alpha$, so that
Lemma~\ref{L:smdecderiv} can not be directly applied.
Therefor, we have to estimate some additional terms in combination
with Theorem~\ref{T:fullrobust} to finally get the desired improved
estimates.

\begin{theo}\label{T:uimprobust}
  Let the assumptions of Theorem~\ref{T:fullrobust}
  be fulfilled. Further, we assume located controls and
  $A=\Omega_U$ up to a set of measure zero
  (measure condition on the whole domain).
  If $\kappa < 1$, we additionally require
  the $\popt_\alpha$-measure condition \eqref{E:struct2}.
  (For $\kappa \ge 1$, this condition is automatically met as shown in
  \cite[Lemma~26]{dissnvd}.)

  Then, for $\alpha >0$ sufficiently small,
  $d_0:=d_0(\uopt_\alpha)$ given as in Theorem~\ref{T:discrerr}, and 
  $C_{ab}$ defined in Lemma~\ref{L:smdecderiv} it holds
  \begin{multline*}
       \norm{\uopt_0-\uopt_d}_U^2 
     + \norm{\uopt_0-\uopt_d}_{L^1(\Omega_U)} \\
       \le C \left( \alpha + h^2 + k^2
      \max(1,C_{ab},\alpha^{\kappa/2-1}) 
      \right)^\kappa  (1+d_0^\kappa)
  \end{multline*}
  for the error in the control.
\end{theo}
\begin{proof}
   Let us recall the estimate \eqref{E:fullrobust1} from the
   proof of Theorem~\ref{T:fullrobust}, i.e.,
   \begin{equation*}   
   \begin{aligned}
      &\norm{\uopt_0-\uopt_d}^{1+1/\kappa}_{L^1(A)}
        + \alpha \norm{\uopt_0-\uopt_d}_U^2
        + \norm{y(\uopt_0)-y(\uopt_d)}_I^2\\
      &\le C\Big(
             -\dual Bp(y(\uopt_d)-y_d) 
                + \dual Bp_{kh}(y_{kh}(\uopt_d)-y_d)
                               + \alpha\uopt_0,\uopt_0-\uopt_d
                   \Big)_U,
   \end{aligned}
   \end{equation*}
   which we rearrange as follows:
   \begin{equation}\label{E:follrobust1rea}
   \begin{aligned}
      &\norm{\uopt_0-\uopt_d}^{1+1/\kappa}_{L^1(A)}
        + \alpha \norm{\uopt_0-\uopt_d}_U^2
        + \norm{y_{kh}(\uopt_0)-y_{kh}(\uopt_d)}_I^2\\
      &\le C\Big(
       \underbrace{
          -\dual Bp(y(\uopt_0)-y(\uopt_\alpha)) 
        }_I
       \underbrace{
        -\dual Bp(y(\uopt_\alpha)-y_d) 
                + \dual Bp_{kh}(y_{kh}(\uopt_\alpha)-y_d)
        }_{IIa} \\
       &\quad\quad
       \underbrace{
          + \alpha\uopt_0
        }_{IIb}
       \underbrace{
         + \dual Bp_{kh}(y_{kh}(\uopt_0)-y_{kh}(\uopt_\alpha))
        }_{III}
                  ,\uopt_0-\uopt_d \Big)_U.
   \end{aligned}
   \end{equation}
For term $III$, we use the optimality conditions together with
Cauchy's inequality to get
  \begin{multline*}
     (y_{kh}(\uopt_0)-y_{kh}(\uopt_\alpha))
      ,y_{kh}(\uopt_0)-y_{kh}(\uopt_d) )_I\\
     \le C \norm{  y_{kh}(\uopt_0)-y_{kh}(\uopt_\alpha) }_I^2
          + \frac{1}{16}\norm{ y_{kh}(\uopt_0)-y_{kh}(\uopt_d) }_I^2,
  \end{multline*}
and move the latter addend to the left-hand side
of \eqref{E:follrobust1rea}.
We split the former addend with the help of
\eqref{E:ykherr} from Lemma~\ref{L:yDiscrStabh} and the regularization
errors \eqref{E:regconvumeas2} and \eqref{E:regconvymeas} 
to obtain with the help of Young's inequality
  \begin{equation}\label{E:ykhu0ua}
  \begin{aligned}
  \norm{  y_{kh}(\uopt_0)-y_{kh}(\uopt_\alpha) }_I^2
   &\le C(\norm { (\hat y_{kh}-\hat y)(\uopt_0-\uopt_\alpha) }_I
       + \norm{y(\uopt_0)-y(\uopt_\alpha)}_I )^2\\
   &\le C( (k+h^2)\alpha^{\kappa/2} + \alpha^\frac{1+\kappa}{2} )^2\\
   &\le C (k+h^2)^{2(\kappa+1)} + C\alpha^{1+\kappa}
  \end{aligned}
  \end{equation}
where $\hat y_{kh}$ and $\hat y$ denote the solution operators
for the state equation with initial value zero. 

For $IIb$, we invoke again Young's inequality and the inclusion
$\uopt_0\in\Uad\subset L^\infty$ to get the estimate
  \begin{equation*}
    \alpha (\uopt_0, \uopt_0-\uopt_d)_U
     \le C\alpha \norm{\uopt_d-\uopt_0}_{L^1(\Omega_U)}
           \le C\alpha^{\kappa+1} + \frac{1}{16}
              \norm{\uopt_d-\uopt_0}_{
                    L^1(\Omega_U)}^{1+1/\kappa}.
  \end{equation*}
  We now move the second summand to the left
  of \eqref{E:follrobust1rea}
  since $A=\Omega_U$ up to a set of measure zero.

The addend $IIa$ can be rewritten and estimated with again
the help of Young's inequality to get
\begin{equation*}
\begin{aligned}
   &\Big(
          -\dual Bp(y(\uopt_\alpha)-y_d) 
               + \dual Bp_{kh}(y_{kh}(\uopt_\alpha)-y_d)
          ,\uopt_0-\uopt_d \Big)_U\\
   &\le C\Big(
         \dual B(p_{kh}-p)(y(\uopt_\alpha)-y_d)
          + \dual Bp_{kh}(y_{kh}(\uopt_\alpha) -y(\uopt_\alpha))
          ,\uopt_0-\uopt_d \Big)_U\\  
   &\le C
      \norm{ \dual B(p_{kh}-p)(y(\uopt_\alpha)-y_d)
          + \dual Bp_{kh}(y_{kh}(\uopt_\alpha) -y(\uopt_\alpha))
           }_{L^\infty(\Omega_U)}^{1+\kappa} \\
   &\quad\quad + \frac{1}{16} \norm{\uopt_0-\uopt_d  
           }_{L^1(\Omega_U)}^{1+1/\kappa} .
\end{aligned}
\end{equation*}
The last addend can now be moved to the left
of \eqref{E:follrobust1rea}.

For summand $I$, we add an additional term to get
  \begin{multline*}
     \Big( -\dual Bp(y(\uopt_0)-y(\uopt_\alpha)) 
                  ,\uopt_0-\uopt_d \Big)_U\\
   = \Big( \dual B(p_{kh}-p)(y(\uopt_0)-y(\uopt_\alpha)) 
          -\dual Bp_{kh}(y(\uopt_0)-y(\uopt_\alpha))  
                 ,\uopt_0-\uopt_d \Big)_U.
  \end{multline*}
We estimate the second addend with the help of the
regularization error \eqref{E:regconvymeas} as 
  \begin{equation*}
   \Big( y(\uopt_0)-y(\uopt_\alpha), 
        y_{kh}(\uopt_0)-y_{kh}(\uopt_d) \Big)_I
   \le C\alpha^{1+\kappa} 
        + \frac{1}{16} \norm{ y_{kh}(\uopt_0)-y_{kh}(\uopt_d)}_I^2,
  \end{equation*}
and move the second addend to the left
of \eqref{E:follrobust1rea}.
For the remaining addend, we use again the above mentioned results
and the estimate \eqref{E:bangU21} to obtain
  \begin{equation*}
  \begin{aligned}
    &\Big( \dual B(p_{kh}-p)(y(\uopt_0)-y(\uopt_\alpha)) 
                 ,\uopt_0-\uopt_d \Big)_U\\
    &= \Big(y(\uopt_0)-y(\uopt_\alpha),
              (\hat y_{kh}-\hat y)(\uopt_0-\uopt_d)\Big)\\
    &\le C \norm{y(\uopt_0)-y(\uopt_\alpha)  }_I^2
       + C\norm{(\hat y_{kh}-\hat y)(\uopt_0-\uopt_d)}_I^2\\
    &\le C\alpha^{1+\kappa} + C(k+h^2)^2 \norm{\uopt_0-\uopt_d}_U^2 \\
    &\le C\alpha^{1+\kappa} + C(k+h^2)^2 
             \norm{\uopt_0-\uopt_d}_{L^1(\Omega_U)}\\
    &\le C\alpha^{1+\kappa} + C(k+h^2)^{2(\kappa+1)}  
        + \frac{1}{16} \norm{\uopt_0-\uopt_d}
                  _{L^1(\Omega_U)}^{1+1/\kappa}
  \end{aligned}
  \end{equation*}
and move the last term to the left of \eqref{E:follrobust1rea}.

Collecting all previous estimates, we
with $L:=L^\infty(I,\H)$ obtain
  \begin{equation*}
   \begin{aligned}
      &\norm{\uopt_0-\uopt_d}^{1+1/\kappa}_{L^1(A)}
        + \alpha \norm{\uopt_0-\uopt_d}_U^2
        + \norm{y_{kh}(\uopt_0)-y_{kh}(\uopt_d)}_I^2\\
      &\le C\Big(
          \alpha^{\kappa + 1}
          +(k+h^2)^{2(\kappa+1)}
         + \norm{ (p_{kh}-p)(y(\uopt_\alpha)-y_d)
           }_{L}^{1+\kappa} \\
      &\quad\quad
          + \norm{ p_{kh}(y_{kh}(\uopt_\alpha)
                   - \mathcal P_{Y_k} y(\uopt_\alpha))
           }_{L}^{1+\kappa} 
         \Big).
\end{aligned}
\end{equation*}

Note that we introduced the orthogonal projection
$\mathcal P_{Y_k}$ in the last
addend, which is possible due to the definition 
of the fully discrete adjoint equation \eqref{E:AdjDiscrh}.
Furthermore, we used stability of $\dual B$ for located controls.

We combine the previous estimate with
the (improved) adjoint error estimate from Lemma~\ref{L:errphklool2},
the adjoint stability from Lemma~\ref{L:AdjDiscrStabh},
and the superconvergence result from Lemma~\ref{L:ImprRatekh},
making use of the regularity given in Lemma~\ref{L:regOCPposalpha},
to get
\begin{equation}\label{E:bestrobust}
   \begin{aligned}
      &\norm{\uopt_0-\uopt_d}^{1+1/\kappa}_{L^1(A)}
        + \alpha \norm{\uopt_0-\uopt_d}_U^2
        + \norm{y_{kh}(\uopt_0)-y_{kh}(\uopt_d)}_I^2\\
      &\le C\Big(
          \alpha +h^2d_0 + k^2(1+d_1^+(\uopt_\alpha))
            \Big)^{1+\kappa}.
\end{aligned}
\end{equation}
With the help of the estimate given in Lemma~\ref{L:smdecderiv}
for $p=2$, i.e.,
\begin{equation*}
  \norm{\partial_t \uopt_\alpha}_{L^2(\Omega_U)}
     \le C\max(C_{ab},\alpha^{\kappa/2-1}), 
\end{equation*}
we conclude that for $\alpha > 0$ sufficiently small it holds
\begin{equation}\label{E:d1plusdecay}
   d_1^+(\uopt_\alpha)  \le C  
       +C\max(C_{ab},\alpha^{\kappa/2-1}). 
\end{equation}
In conclusion, we get
\begin{equation*}
   \begin{aligned}
      &\norm{\uopt_0-\uopt_d}^{1+1/\kappa}_{L^1(A)}
        + \alpha \norm{\uopt_0-\uopt_d}_U^2
        + \norm{y_{kh}(\uopt_0)-y_{kh}(\uopt_d)}_I^2\\
      &\le C\Big(
          \alpha +h^2d_0+ k^2
        \max(1,C_{ab},\alpha^{\kappa/2-1})
            \Big)^{1+\kappa}.
\end{aligned}
\end{equation*}

Finally, recall that the $U$ error in the control can be derived 
from the corresponding $L^1$ error using the estimate \eqref{E:bangU21}.
\end{proof}

From the previous theorem we get coupling rules
for $\alpha$ and $k$, always with $\alpha = h^2$,
and convergence rates, which are shown in the following table.
\begin{table}[hbt]
\begin{center}
\begin{tabular}{c|c|c}
\hline
$\alpha=$ 
& $\norm{\uopt_d-\uopt_0}_{L^1(\Omega_U)}\le C\dots$
& if \\
\hline
$k^{4/(4-\kappa)}$
& $\alpha^\kappa = h^{2\kappa} = k^{4\kappa/(4-\kappa)}$
& $\kappa < 2$ \\
$k^2$
& $\alpha^\kappa = h^{2\kappa} = k^{2\kappa}$
& $\kappa \ge 2$\\
\hline
\end{tabular}
\end{center}
\caption{Coupling and convergence implied by Theorem~\ref{T:uimprobust}.}
\label{tab:coupr}
\end{table}

Note that in any case we get a better rate than
$k^\kappa$ proven in Theorem~\ref{T:fullrobust2}.

\begin{coro}\label{C:uimprobust}
 Let the assumptions of the previous Theorem hold.
 For the adjoint and the projected state we have the error estimate
\begin{multline*}
 \norm{\popt_0-\popt_d}_{L^\infty(I,\H)}
        + \norm{\yopt_0-\pi_{P_k^*}\yopt_d}_I\\
 \le C\alpha^{\max(\frac{\kappa+1}{2},\kappa)} 
     +C\Big( h^2d_0 
      +k^2\max(1,C_{ab},\alpha^{\kappa/2-1}) 
         \Big)^{\min(1,\frac{\kappa+1}{2})}.
\end{multline*}
\end{coro}
\begin{proof}
Inspecting the proof of Corollary~\ref{C:nonrobest},
we obtain the estimate
\begin{equation*}
      \norm{\popt_\alpha-\popt_d}_{L^\infty(I,\H)}
        + \norm{\yopt_\alpha-\pi_{P_k^*}\yopt_d}_I\\
    \le C(k^2d_1^+ + h^2d_0 + \norm{y_{kh}(\uopt_\alpha)-\yopt_d}_I).
\end{equation*}
To estimate the last addend, let us first combine
the estimate \eqref{E:bestrobust} from the proof of 
Theorem~\ref{T:uimprobust} with Remark~\ref{R:02alpha} to get
\begin{equation*}
   \norm{\uopt_\alpha-\uopt_d}_{L^1(A)}
      \le C\Big(
          \alpha +h^2d_0 + k^2(1+d_1^+(\uopt_\alpha))
            \Big)^{\kappa}.
\end{equation*}
With this estimate, we now follow the proof of 
Theorem~\ref{T:fullrobust2} from the entry point 
\eqref{E:supconvykhualpha} onwards. We obtain
\begin{multline*}
      \norm{\popt_\alpha-\popt_d}_{L^\infty(I,\H)}^2
       + \norm{\yopt_\alpha-\pi_{P_k^*}\yopt_d}_I^2\\
    \le C\left(\left(h^2d_0 + k^2 d_1^+\right)^2 
          + \left(h^2d_0 + k^2 d_1^+\right)
      \left(\alpha+h^2d_0+k^2 \left(1+d_1^+\right)\right)^\kappa
         \right).
\end{multline*}
With Young's inequality, the regularization error \eqref{E:regconvymeas},
property \eqref{E:dualTlinfty}, and the decay estimate 
\eqref{E:d1plusdecay}, we finally get the claim.
\end{proof}

\section{Numerics}

We will now consider a test example in order to finally validate
numerically the theoretical results.

As we have previously said, we solve numerically the regularized problem
\eqref{OCPkh} for some $\alpha > 0$ as an approximation of the 
limit problem \eqref{OCPl}.
Thus, we have the influence of two errors: 
The regularization error in dependence of the parameter
$\alpha > 0$ and the discretization error due to space and time 
approximation.
The second error depends on the fineness of the space and time grid,
respectively, thus on the parameters $h$ and $k$.

We do not investigate the 
time discretization error for fixed 
positive $h$ and $\alpha$ by taking $k\to 0$, since this can be found in
\cite{DanielsHinzeVierling}.
The numerical behavior of the error 
if $h\to 0$, again for fixed $\alpha > 0$ but now
with fixed $k$ instead of $h$ is discussed 
in \cite[section~3.1.2]{dissnvd}. 
The regularization error
for fixed small discretization parameters $k$ and $h$ in dependence
of the parameter $\kappa$ from the measure condition \eqref{E:struct}
if $\alpha \to 0$ can be found in \cite{daniels} 
or \cite[section~3.2]{dissnvd}.

Here, we only report on the coupling of
regularization and discretization parameters as
proposed by Theorem~\ref{T:uimprobust} and Table~\ref{tab:coupr}.

We make use of the fact that instead of the linear control
operator $B$, given by \eqref{E:B}, we can also use an \emph{affine linear}
control operator
\begin{equation}\label{E:Btilde}
\tilde B: U\rightarrow L^2(I,\Vd)\,,\quad u\mapsto g_0 + Bu
\end{equation}
where $g_0$ is a fixed function.
If we assume that $g_0$ is an element of the space 
$H^1(I,L^2(\Omega))$ with 
$g_0(0)\in\V$ and $g_0(T)\in\V$, the preceding theory remains valid
since $g_0$ can be interpreted as a modification of $y_d$.

For the limit problem \eqref{OCPl}, we consider a test example
which is a bang-bang problem with $\meas(A^c) = 0$ and $\kappa=1$
in Assumption~\ref{A:sourcestruct}. 

With a space-time domain $\Omega\times I := (0,1)^2 \times (0,0.5)$,
we consider a located control function $\uopt$
and a constant $a:=2$, not to be confused with the lower
bound $a_1$ of the admissible set $\Uad$ defined below.
This constant $a$ influences the number 
of switching points between the active and inactive set. 
Furthermore, we define the functions
\[
      g_1(x_1,x_2) := \sin(\pi x_1)\sin(\pi x_2)\,,
\]
\[
      w_a(t,x_1,x_2) 
          := \cos\left(\frac tT\,2\pi a\right)\cdot g_1(x_1,x_2)\,,
\]
and choose an optimal adjoint state
\[
    \popt := \frac{-T}{2\pi a} \sin\left( \frac tT 2\pi a\right) g_1\,,
\]
which is nonzero almost everywhere, and since 
\[
    -\partial_t \popt - \Delta \popt
       = \cos\left( \frac tT 2\pi a\right) g_1 
          -\frac{T}{2\pi a} \sin\left( \frac tT 2\pi a\right) 2\pi^2 g_1 
               = \yopt - y_d \,,
\]
we get the function $y_d$ by taking $\yopt$ as
\begin{equation}\label{E:tht2y}
      \yopt(t,x_1,x_2) := w_a(t,x_1,x_2)\,.
\end{equation}

From the relation \eqref{E:bangbang}
we conclude that the optimal control is given by
\[
    \uopt =             \begin{cases}
        a_1    & \text{if $\dual B\popt > 0$},\\
        b_1    & \text{if $\dual B\popt < 0$}
                        \end{cases}
\]

Note that $\dual B\popt (t) = (g_1,\popt(t))_{\H}$,
the initial value of the optimal state $\yopt$ is
\[
      y_0(x_1,x_2) = \yopt(0,x_1,x_2) = g_1(x_1,x_2)\,,
\]
and $(g_1,g_1)_{\H} = 0.25$. We obtain
\begin{equation}\label{E:tht2g0}
      g_0= g_1 2 \pi \left( -\frac aT \sin\left(\frac tT\,2\pi a\right)
          + \pi\cos\left(\frac tT\,2\pi a\right) \right)-B\uopt\,,
\end{equation}
and finally define the bounds of the admissible set $\Uad$ as
$a_1:=0.2$ and $b_1:=0.4$.

Since $\kappa=1$ in this example, 
we conclude with Theorem~\ref{T:uimprobust},
Corollary~\ref{C:uimprobust}, and the second
line of Table~\ref{tab:coupr} the estimate
\begin{multline}\label{E:couplingest}
    \norm{\uopt_0-\uopt_d}_U^2
    + \norm{\uopt_0-\uopt_d}_{L^1(A)} 
    + \norm{\popt_0-\popt_d}_{L^\infty(I,\H)}
        + \norm{\yopt_0-\pi_{P_k^*}\yopt_d}_I\\
   \le C(\alpha + h^2 + k^{4/3}).
\end{multline}
Consequently, we set $\text{Nh}=(2^{\ell}+1)^2$, 
$\text{Nk}=(2^{3/2\ell+1}+1)$, and 
$\alpha = 2^{-2\ell}$ with $\ell=1,2,3,4,5,6$, to
obtain second order convergence with respect to $h$ 
in \eqref{E:couplingest}.

We solve \eqref{OCPkh} numerically with the above data 
using a fixed-point iteration for equation \eqref{FONCkh}. 
Each fixed-point iteration is initialized with the 
starting value 
$u_{kh}^{(0)}:=a_1$ which is the lower bound of the admissible set.
As a stopping criterion for the fixed-point iteration, we require for
the discrete adjoint states belonging to the current
and the last iterate that
\[
    \norm{ \dual B \left(p_{kh}^{(i)} - p_{kh}^{(i-1)}\right)
         }_{L^\infty(\Omega\times I)} < t_0
\]
where $t_0:=10^{-5}$ is a prescribed threshold.

The results are given in Tables~\ref{tab:ex3u}, \ref{tab:ex3y},
\ref{tab:ex3yp}, and \ref{tab:ex3p}. 
We also refer to Figure~\ref{fig:ex3}.

As one can see from the tables, the coupling shows the expected
behavior for the error in the optimal control, projected state,
and adjoint state.

Note that for the state $\yopt$, we
observe convergence of order $3/2$, which means by the coupling
from above ($k=h^{3/2}$) first order convergence in $k$.
Thus, it is in accordance with our expectation since
the state is discretized piecewise constant in time.
This is depicted in Table~\ref{tab:ex3y}. 

A better and second order convergent approximation of the state 
is given by the projection $\pi_{P_k^*}y_{kh}$ 
of the computed discrete state $y_{kh}$, see Corollary~\ref{C:nonrobest}
and for the corresponding numerical results see Table~\ref{tab:ex3yp}. 
This better approximation of the state can be obtained without
further numerical effort: One only has to interpret the 
vector containing the values of $y_k$ on each interval $I_m$
as a vector of linearly-in-time linked values on the gridpoints
of the dual grid $t_1^* < \dots < t_M^*$.

Figure~\ref{fig:ex3} illustrates the convergence of $u_{kh}$ to $\uopt$. 
Note that the intersection points between the inactive set 
$\mathcal I_{kh}:=\twoset{t\in I}{a < u_{kh}(t) < b}$ and the active set 
$\mathcal A_{kh}:=I\backslash \mathcal I_{kh}$ 
need not coincide with the time grid points since we use 
variational discretization for the control.

\begin{table}[hbt]
\begin{center}
\begin{tabular}{ccccccc}
\hline
& $\norm{\uopt-u_{kh}}$
& $\norm{\uopt-u_{kh}}$
& $\text{EOC}$
& $\text{EOC}$\\
$\ell$ 
& $L^1(I,\mathbb R)$ 
& $L^2(I,\mathbb R)$ 
& $L^1$
& $L^2$\\
\hline
 1 &  0.05208333 &  0.10206207 &    /  &    /  \\
 2 &  0.05156250 &  0.10155048 &  0.01 &  0.01 \\
 3 &  0.01551730 &  0.05249039 &  1.73 &  0.95 \\
 4 &  0.00395214 &  0.02696386 &  1.97 &  0.96 \\
 5 &  0.00100074 &  0.01375946 &  1.98 &  0.97 \\ 
 6 &  0.00026290 &  0.00704586 &  1.93 &  0.97 \\
\hline
\end{tabular}
\end{center}
\caption{
     Errors and h-EOC in the control ($\alpha = k^{4/3} = h^2$).}
\label{tab:ex3u}
\end{table}

\begin{table}[hbt]
\begin{center}
\begin{tabular}{ccccccc}
\hline
& $\norm{\yopt-y_{kh}}$
& $\norm{\yopt-y_{kh}}$
& $\norm{\yopt-y_{kh}}$
& $\text{EOC}$
& $\text{EOC}$
& $\text{EOC}$\\
$\ell$ 
& $L^1(I,L^1(\Omega))$
& $L^2(I,L^2(\Omega))$
& $L^\infty(I,L^\infty(\Omega))$
& $L^1$
& $L^2$
& $L^\infty$\\
\hline
 1 &  0.04168338 &  0.14344433 &  0.77006182 &    /  &    /  &    /  \\
 2 &  0.02298795 &  0.05061771 &  0.24946457 &  0.86 &  1.50 &  1.63 \\
 3 &  0.00877452 &  0.01795226 &  0.08863801 &  1.39 &  1.50 &  1.49 \\
 4 &  0.00314952 &  0.00624197 &  0.02943581 &  1.48 &  1.52 &  1.59 \\
 5 &  0.00111871 &  0.00218973 &  0.00994956 &  1.49 &  1.51 &  1.56 \\
 6 &  0.00039580 &  0.00077075 &  0.00339060 &  1.50 &  1.51 &  1.55 \\
\hline
\end{tabular}
\end{center}
\caption{
     Errors and h-EOC in the state ($\alpha = k^{4/3} = h^2$).}
\label{tab:ex3y}
\end{table}

\begin{table}[hbt]
\begin{center}
\begin{tabular}{ccccccc}
\hline
& $\norm{\yopt-\pi_{P_k^*} y_{kh}}$
& $\norm{\yopt-\pi_{P_k^*} y_{kh}}$
& $\norm{\yopt-\pi_{P_k^*} y_{kh}}$
& $\text{EOC}$
& $\text{EOC}$
& $\text{EOC}$\\
$\ell$ 
& $L^1(I,L^1(\Omega))$
& $L^2(I,L^2(\Omega))$
& $L^\infty(I,L^\infty(\Omega))$
& $L^1$
& $L^2$
& $L^\infty$\\
\hline
 1 &  0.03984472 &  0.12699052 &  0.67616861 &    /  &    /  &    /  \\
 2 &  0.01063414 &  0.02423705 &  0.15855276 &  1.91 &  2.39 &  2.09 \\
 3 &  0.00235558 &  0.00482756 &  0.02588151 &  2.17 &  2.33 &  2.61 \\
 4 &  0.00059757 &  0.00116777 &  0.00526572 &  1.98 &  2.05 &  2.30 \\
 5 &  0.00015345 &  0.00029551 &  0.00128779 &  1.96 &  1.98 &  2.03 \\
 6 &  0.00003968 &  0.00007581 &  0.00032323 &  1.95 &  1.96 &  1.99 \\
\hline
\end{tabular}
\end{center}
\caption{
     Errors and h-EOC in the projected state ($\alpha = k^{4/3} = h^2$).}
\label{tab:ex3yp}
\end{table}

\begin{table}[hbt]
\begin{center}
\begin{tabular}{ccccccc}
\hline
& $\norm{\popt-p_{kh}}$
& $\norm{\popt-p_{kh}}$
& $\norm{\popt-p_{kh}}$
& $\text{EOC}$
& $\text{EOC}$
& $\text{EOC}$\\
$\ell$ 
& $L^1(I,L^1(\Omega))$
& $L^2(I,L^2(\Omega))$
& $L^\infty(I,L^\infty(\Omega))$
& $L^1$
& $L^2$
& $L^\infty$\\
\hline
 1 &  0.00175355 &  0.00559389 &  0.02497779 &    /  &    /  &    /  \\
 2 &  0.00052886 &  0.00120225 &  0.00578048 &  1.73 &  2.22 &  2.11 \\
 3 &  0.00012807 &  0.00026289 &  0.00128201 &  2.05 &  2.19 &  2.17 \\
 4 &  0.00003156 &  0.00006214 &  0.00028508 &  2.02 &  2.08 &  2.17 \\
 5 &  0.00000786 &  0.00001530 &  0.00006829 &  2.01 &  2.02 &  2.06 \\
 6 &  0.00000195 &  0.00000377 &  0.00001649 &  2.01 &  2.02 &  2.05 \\
\hline
\end{tabular}
\end{center}
\caption{
     Errors and h-EOC in the adjoint state ($\alpha = k^{4/3} = h^2$).}
\label{tab:ex3p}
\end{table}

\begin{figure}
 \centering  
 \subfloat[$\ell=3$]{\includegraphics[trim=25mm 75mm 15mm 91mm,
        clip,width=0.3\textwidth]{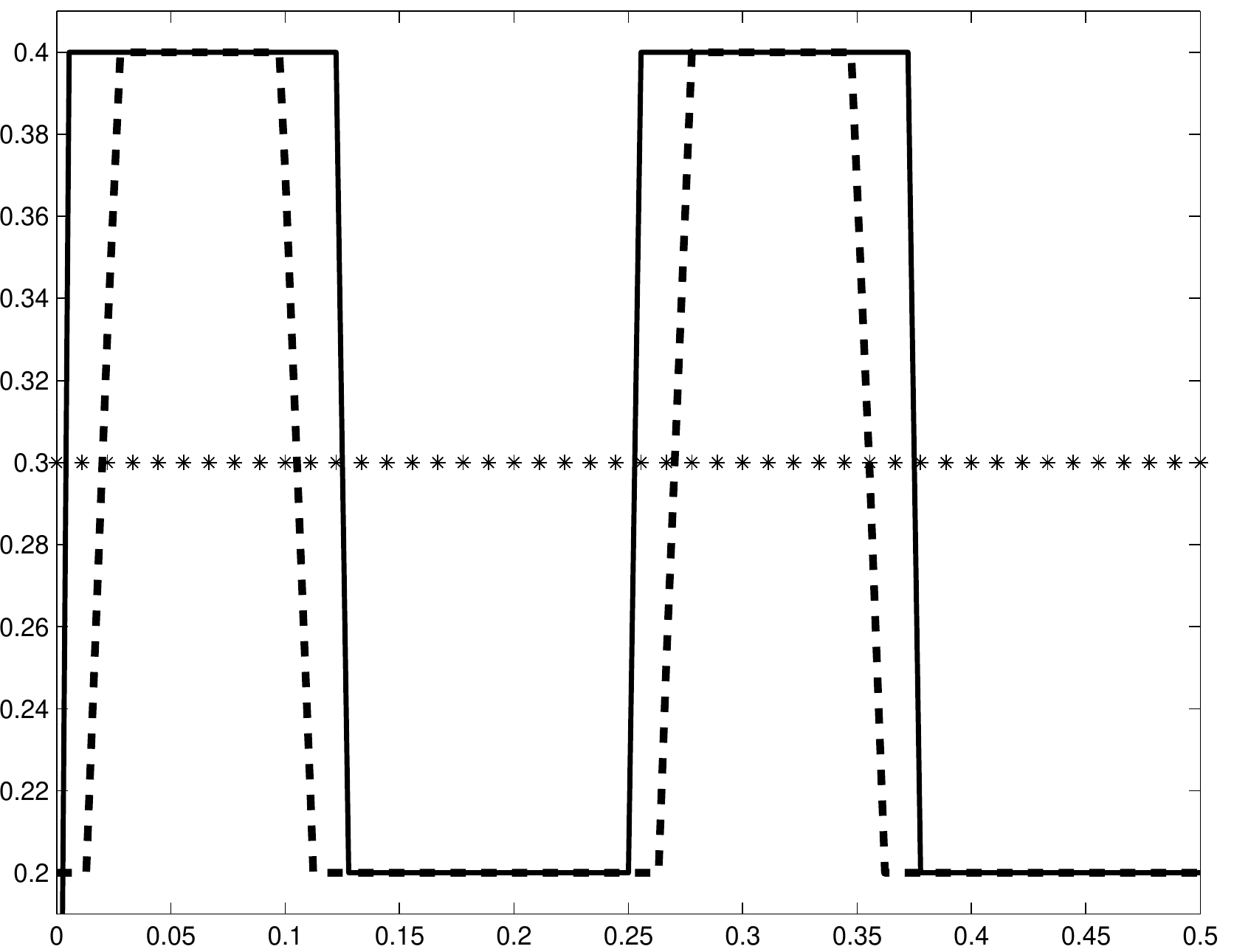}}
 \subfloat[$\ell=4$]{ \includegraphics[trim=25mm 75mm 15mm 91mm,
        clip,width=0.3\textwidth]{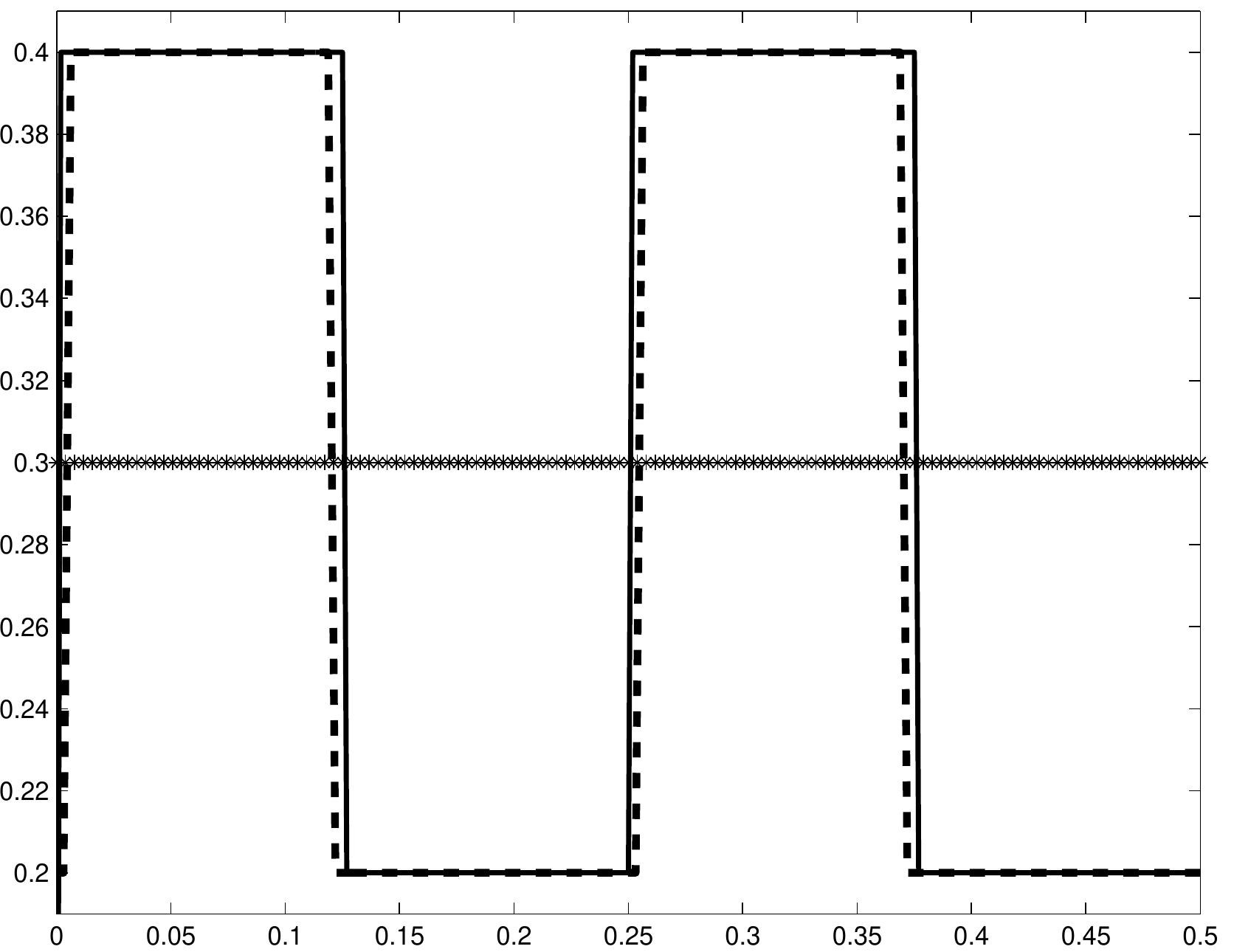}}
 \subfloat[$\ell=5$]{ \includegraphics[trim=25mm 75mm 15mm 91mm,
        clip,width=0.3\textwidth]{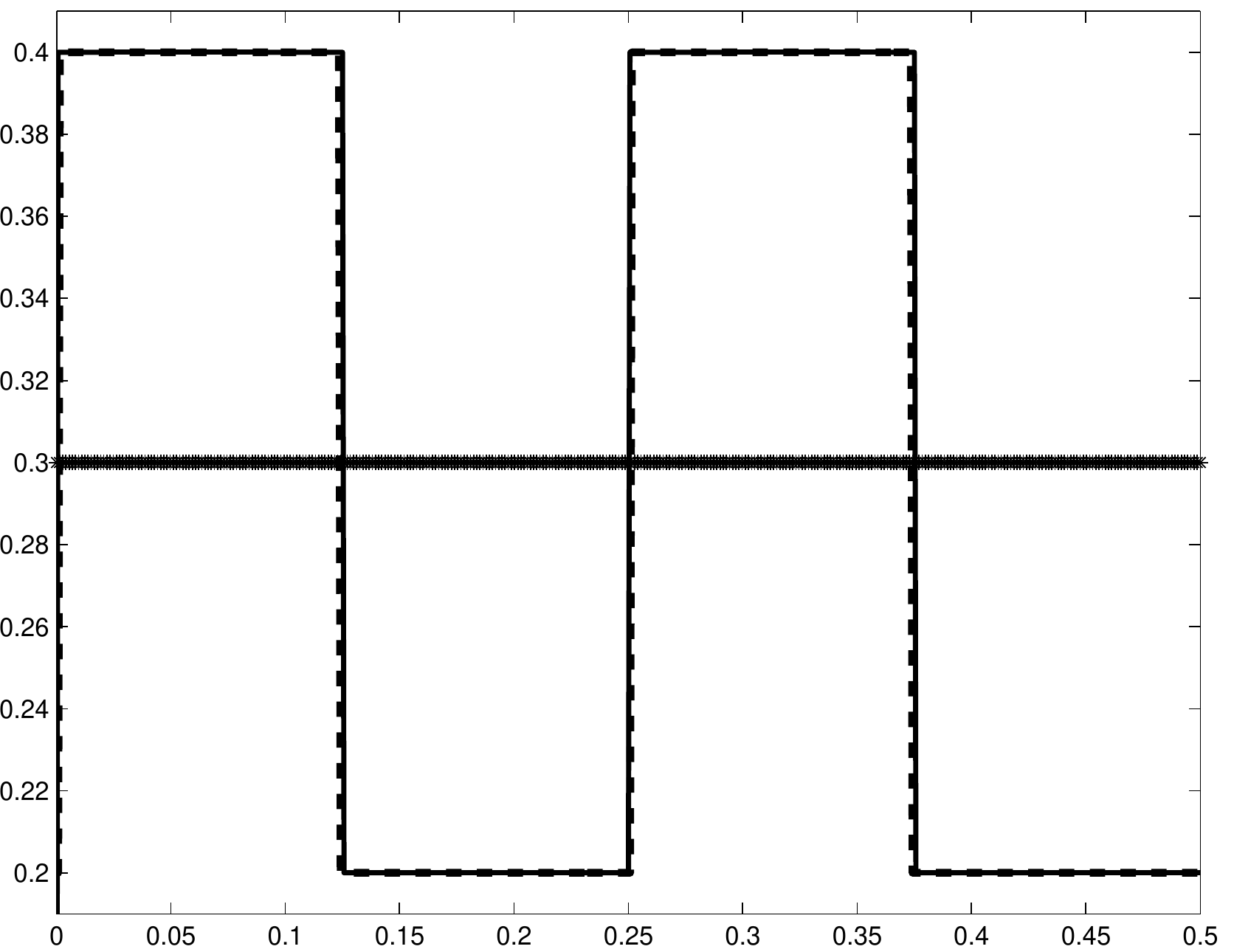}}
    \caption{
         Optimal control $\uopt$ (solid) and computed counterpart
       $u_{kh}$ (dashed) over time after level $\ell$
        ($\alpha = k^{4/3} = h^2$).  }
    \label{fig:ex3}
\end{figure}

Let us mention that the convergence of the fixed-point iteration 
is in general guaranteed only for values of $\alpha$ not too small.
This is an immediate consequence
of Banach's fixed-point theorem in combination with \eqref{FONCkh}.
In the numerical examples we considered, no convergence problems
occurred, even for very small values of $\alpha$. This might be
due to the fact that we consider controls which ``live'' in one
space dimension only. For higher dimensions, the situation 
is more delicate. 
There, the application of semismooth
Newton methods has turned out to be fruitful, see 
\cite{HinzeVierling2012} for its
numerical analysis in the case of variational discretization of 
elliptic optimal control problems.

\printbibliography
\end{document}